\documentclass[11pt,a4paper,reqno]{amsart}
\usepackage[utf8]{inputenc}
\usepackage[T1]{fontenc}
\usepackage{amsfonts,amsmath,graphics,epsfig,latexsym,times}
\usepackage{ae,aecompl}
\usepackage[pdf]{pstricks}
\usepackage{pstricks-add}
\usepackage[rotate,arc,all]{xy}

\newcommand{\eq}[1][r]
   {\ar@<-3pt>@{-}[#1]
    \ar@<-1pt>@{}[#1]|<{}="gauche"
    \ar@<+0pt>@{}[#1]|-{}="milieu"
    \ar@<+1pt>@{}[#1]|>{}="droite"
    \ar@/^2pt/@{-}"gauche";"milieu"
    \ar@/_2pt/@{-}"milieu";"droite"}
\newif\ifpix \pixtrue

\usepackage[colorlinks=true]{hyperref}
\hypersetup{ pdfstartview=FitH }

\numberwithin{equation}{section}

\def\TT{{\mathbb{T}}}
\def\JJ{{\mathbb{J}}}

\def\ZZ{{\mathbb{Z}}}    \def\QQ{{\mathbb{Q}}} \def\CC{{\mathbb{C}}}
\def\RR{{\mathbb{R}}}  \def\TT{{\mathbb{T}}}   \def\TT{{\mathbb{T}}} 
\def\cO{{\mathcal{O}}}
\def\cG{{\mathcal{G}}}
\def\cH{{\mathcal{H}}}
\def\cJ{{\mathcal{J}}}

\def\fK{{\mathfrak{K}}}

\def\cL{{\mathcal{L}}}

\def\cK{{\mathcal{K}}}
\def\cV{{\mathcal{V}}}

\def\cH{{\mathcal{H}}}
\def\cU{{\mathcal{U}}}

\def\cW{{\mathcal{W}}}
\def\AC{\mathcal{AC}}
\def\Ham{\mathrm{Ham}}

\renewcommand{\epsilon}{\varepsilon}
\newcommand{\ip}[1]{\langle #1 \rangle}

\newcommand{\vol}{\mathrm{vol}}
\newcommand{\voltilde}{\widetilde{\mathrm{vol}}}

\newcommand{\id}{\mathrm{id}}
\newcommand{\Diff}{\mathrm{Diff}}

\newcommand{\Ric}{\mathrm{Ric}}

\newcommand{\CP}{\mathbb{CP}}

\newcommand{\del}{\partial}

\newcommand{\vform}{{\mathrm dv}}
\def\ra{\rightarrow}
\def\kt{\mathfrak{t}}
\def\g{\mathfrak{g}}
\def\pol{P}
\def\bT{\mathbb{T}}
\def\tr{{\mathrm tr}}

\def\q{{\bf{q}}}

\newtheorem{lemma}[subsubsection]{Lemma}
\newtheorem{prop}[subsubsection]{Proposition}
\newtheorem{cor}[subsubsection]{Corollary}
\newtheorem{theo}[subsubsection]{Theorem}
\newtheorem{theointro}{Theorem}

\newtheorem{corintro}[theointro]{Corollary}

\theoremstyle{definition}
\newtheorem{dfn}[subsubsection]{Definition}
\newtheorem{rmk}[subsubsection]{Remark}
\newtheorem{example}[subsubsection]{Example}
\newtheorem{examples}[subsubsection]{Examples}

\newtheorem*{rmk*}{Remark}
\newtheorem*{rmks*}{Remarks}
\newtheorem{rmks}[subsubsection]{Remarks}
 \newcommand*{\quot}[2]%
{\ensuremath{%
   \raisebox{.35ex}{\ensuremath{#1}}\big/\raisebox{-.35ex}{\ensuremath{#2}}}}

\title{Hamiltonian stationary Lagrangian fibrations}

\author{Eveline Legendre}\address{Eveline Legendre\\ Universit\'e Paul Sabatier\\
Institut de Math\'ematiques de Toulouse\\ 118 route de Narbonne\\
31062 Toulouse\\ France}
\email{eveline.legendre@math.univ-toulouse.fr}

\author{Yann Rollin}
\address{Yann Rollin, Laboratoire Jean Leray, Universit\'e de Nantes}
\email{yann.rollin@univ-nantes.fr}

\thanks{The authors are partially supported by France ANR project EMARKS No ANR-14-CE25-0010. 
The first author visited the second on a Labex CIMI grant. The authors wish to thank Henri Anciaux for a valuable discussion.}

\begin{document}
{\Huge \sc \bf\maketitle}
\begin{abstract}
Hamiltonian stationary Lagrangian submanifolds (HSLAG)  are a natural generalization of special Lagrangian manifolds (SLAG). 
The latter only make sense on Calabi-Yau manifolds whereas the former are defined for any almost Kähler manifold. Special Lagrangians, and, more specificaly, fibrations by special Lagrangians play an important role in the context of the geometric mirror symmetry conjecture. However, these objects are rather scarce in nature. On the contrary, we show that HSLAG submanifolds, or fibrations, arise quite often. Many examples of HSLAG fibrations  are provided by toric Kähler geometry. In this paper, we obtain a large class of examples by deforming the toric metrics into non toric almost Kähler  metrics, together with HSLAG submanifolds. 
\end{abstract}

\section{Introduction}
Let $M$ be a closed smooth  manifold endowed with a symplectic form
$\omega$ and an almost complex structure $J$. 
If the
bilinear form defined on each tangent space at $m\in M$
$$g_J(v,w)=\omega(v,Jw),\quad  \forall v,w\in T_mM$$
is a Riemannian metric,  we say that $J$ is a \emph{compatible} almost
complex structure on the symplectic manifolds $(M,\omega)$. 
Such a triplet $(M,\omega,J)$ is then called an
\emph{almost K\"ahler manifold}. If in addition $J$
is integrable,  $(M,\omega,J)$ is a \emph{K\"ahler manifold}. The
space of all compatible complex structures on a given symplectic
manifold will be denoted $\AC_\omega$.

Let  $L$ be a closed  manifold and $\ell:L\to M$ a
Lagrangian embedding. In other words,  $\dim M = 2 \dim L$ and $\ell^*\omega = 0$.
 Hamiltonian transformations $u\in\Ham_\omega$ of $(M,\omega)$ act on such
Lagrangian maps by composition on the left
$$u\cdot \ell = u\circ \ell.$$
Sometimes we shall use a different, yet equivalent, point of view where $\Ham_\omega$ leaves $\ell:L\to M$
fixed and acts instead on the space $\AC_\omega$ by 
$$
u\cdot J = u^*J.
$$

Given an almost K\"ahler manifolds $(M,\omega, J)$,
a Lagrangian embedding  is called \emph{Hamiltonian stationary} if
it is a critical point of the volume functional under Hamiltonian
deformations. As a short hand, such an embedding will be called a
\emph{HSLAG embedding} and its image a \emph{HSLAG submanifold} of $(M,\omega,J)$. These Lagrangians were
introduced by Oh~\cite{O90} and may be understood as a natural
generalization  of special Lagrangian manifolds (SLAG), defined in the case of a
Calabi-Yau manifolds.

More concretely, for $J\in\AC_\omega$ and a Lagrangian embedding $\ell
:L\to M$, the
volume $\vol(\ell, J):= \vol (\ell(L),g_J)$ is the volume of $L$ endowed
with the pull-back metric $\ell^*g_J$. By definition of the
the action of $\Ham_\omega$, we have
$$
\vol(u\cdot \ell,J)=\vol(\ell, u\cdot J),
$$
so that finding critical points of the volume functional under the action of $\Ham_\omega$ on
either $\AC_\omega$ or the space of Lagrangian embeddings are
equivalent problems.

The group $\Diff(L)$ of diffeomorphisms of $L$ also acts on the space
of Lagrangian embeddings on the right by 
$$\ell \cdot v =\ell\circ v, \mbox { where } v\in\Diff(L).$$
 However the volume is invariant under this action so that the
problem of finding a stationary Lagrangian embeddind is greatly
underdetermined. Sometimes we may think of Lagrangian submanifold
rather than embeddings to avoid this infinite dimensional degree of
freedom in the equation.

\subsection{The case of a single Lagrangian}
We shall prove the following existence theorem for \emph{rigid} HSLAG embeddings upto small
deformations of the compatible almost complex structure:
\begin{theointro}
\label{theo:A}
Let $(M,\omega,J_0)$  be a closed K\"ahler manifold and $L$ a closed manifold. Let $G$ be the group of
Hamiltonian isometries of $(M,\omega, J_0)$  and
   $\ell:L\to M$  be a rigid HSLAG embedding.

There exists a $G$-invariant neighborhood $W$ of $J_0$ in $\AC_\omega$ and a
 map $\psi:W\to \Ham_\omega$ such that $\psi(J_0)=\id$, with
the property that every $G$-orbit in $W$ contains a representative
$J$,
such that the Lagrangian embedding $\psi(J)\cdot \ell:L\to M$
is HSLAG with respect to $J$.
\end{theointro}
\begin{rmks}
  \begin{enumerate}
\item  As we shall see, the map $\psi$ is in fact continuous once suitable
  Hölder topologies are introduced.

\item The rigidity assumption of Theorem~\ref{theo:A} means that
infinitesimal deformation of the given HSLAG can only come from
Hamiltonian isometries. This property will be introduced precisely
at Definition~\ref{dfn:rigidstable}.

\item If the group of Hamiltonian isometries $G$ is trivial, the above
theorem, in particular the construction of the map $\psi$, follows from the implicit
function theorem. If $G$ is not trivial, the problem becomes obstructed and
this is more tricky. In this
case the solution $J$ in a given $G$-orbit
 comes from the
minimization of a finite dimensional  problem in a $G$-orbit.   
  \end{enumerate}

\end{rmks}

A large pool of examples is provided by toric Kähler manifolds which
typically have a non trivial group $G$. It turns out that
Lagrangian tori of standard toric Kähler manifolds are automatically rigid and
stable HSLAG. More precisely, we have the following result:
\begin{theointro}
\label{theo:torichslag}
  Let $(M,\omega,J_0)$ be a closed toric Kähler manifold endowed with the Guillemin metric. Let $\mu:M\to P$ be the moment map with image the
  moment polytope $P$. Then for any interior point $p$ of $P$, the
  Lagrangian torus $L_p=\pi^{-1}(p)$ is HSLAG, rigid and
  stable in $(M,\omega, J_0)$.
\end{theointro}
A special case of this result concerns  $\CP^n$ endowed with
its standard toric action and the Fubini-Study metric, was proved by Ono in \cite{Ono2007}. 
Theorem~\ref{theo:A} applies in this context for each regular fiber. Notice that in
the one dimensional case, $\CP^1$ is the round sphere and complex dimension
one HSLAG submanifolds are just curves of constant curvature.

\begin{rmk}
Few constructions of HSLAG  submanifolds are
known. 
By variationnal methods, Schoen and Wolfson provided an existence theorem for HSLAG submanifolds
with singularities, on $4$-dimensional almost Kähler manifolds
(cf. \cite{SW1999, SW2001}). 
Joyce, Schoen and Lee constructed microscopic HSLAG tori in almost
Kähler manifolds (cf. \cite{JLS11,Lee12}).  
More recently, Biquard and Rollin constructed HSLAG representative of vanishing cycles in
the smoothing of extremal Kähler surfaces with $\QQ$-Gorenstein
singularities in~\cite{BR15}.  

We point out that the above constructions all deal with \emph{small}
or possibly \emph{singular}
Lagrangian submanifolds,  whereas the construction proposed in the
current paper
addresses the case of \emph{large} and \emph{smooth} Lagrangian submanifolds.
\end{rmk}

\subsection{Toric fibrations}
\label{sec:fibdef}
 SLAG fibrations are a conerstone of geometric mirror
symmetry but seem  to be rather scarce in nature.
Our initial motivation for this work was to exhibit various fibrations by
HSLAG submanifolds and show that contrarily to the SLAG case, they
arise quite often.

A \emph{local Lagrangian toric fibration} of a K\"ahler manifold $(M,\omega,J_0)$ is
a smooth family of Lagrangian maps $\ell_t:L\to M$,  with  parameter
$t\in B(0,r)\subset \RR^n$, for some $r>0$, where $n=\dim_\CC M$ and
$L=\TT^n$ is a real $n$-dimensional torus, such
that the map $B(0,r)\times L\to M$ defined by
$(t,x)\mapsto \ell_t(x)$ is a smooth embedding. 
If every torus $\ell_t:L\to M$ is HSLAG, we say that the local
fibration is HSLAG.

\begin{example}
Toric Kähler manifolds provide natural examples of local
toric HSLAG fibrations according to Theorem~\ref{theo:torichslag}. The fibration does not extend globally since
tori collapse at the boundary of the moment polytope.  

Other interesting (singular) HSLAG toric fibrations also occur
as part of the SYZ mirror symmetry conjecture. In this case, tori are SLAG hence HSLAG.
\end{example}

\begin{rmk}\label{remLagFIB=toric}
The reader may wonder why we restric to Lagrangian fibrations with
toric fibers as above: let $(M,\omega)$
be a $2n$-dimensional manifold endowed with a submersion $\pi:M\to
B^n$, where $B^n$ is an open ball in $\RR^n$, such that the fibers are Lagrangian compact submanifolds of
$M$. An elementary argument shows that in such a situation, the fibers must be $n$-dimensional manifolds
diffeomorphic to the real torus $\TT^n$. Indeed, let $x_1,\cdots,x_n$
be the standard coordinate functions on $\RR^n$ understood as functions
on $M$. These functions induce Hamiltonian vector fields
$Y_1,\cdots,Y_n$ on $M$. Since the fibers are Lagrangian, the vector
fields $Y_j$ must be tangent to the fibers of $\pi$. The fact that the functions
$x_i$ are invariant along the fibers and that $Y_j$ leave the
symplectic form invariant imply that the Lie brackets $[Y_i,Y_j]$
vanish. We have $n$ globally defined linearly independent vector fields
$Y_j$ along each Lagrangian fibers. As they are compact, any connected component of the fibers
must be diffeomorphic the real torus.
\end{rmk}

\begin{dfn}
\label{dfn:equivariant}
Given a HSLAG embedding $\ell:L\to M$ into a Kähler manifold
$(M,J_0,\omega)$, where $L$ and $M$ are closed, let $G_\ell$ be the
subgroup of Hamiltonian isometries $G$ preserving the image of $\ell:L\to M$. In
other words $u\in G_\ell$ if and only if $u\circ\ell(L)=\ell(L)$.
We denote by $G_\ell^o$ the identity component of $G_\ell$.

Let $\ell_t:L\to M$ be a local Lagrangian toric fibration such that $\ell_0=\ell$. Such fibration is
said to be \emph{$G^o_\ell$-invariant} if the action of $G^o_\ell$ preserves
the image of each embedding $\ell_t:L\to M$.  
\end{dfn}
Notice that with the above definition, $G^o_\ell$ acts trivially on
the parameter space $t\in B(0,r)$.

Our next result is an existence theorem for local HSLAG toric
fibrations:
\begin{theointro}
\label{theo:B}
Let $(M,\omega,J_0)$  be a closed K\"ahler manifold and
   $\ell_t:L\to M$ a  $G^o_{\ell_0}$-invariant local toric HSLAG fibration for $t\in
  B(0,r)\subset \RR^n$  such  that $\ell_0$ is rigid.

Then for all sufficiently small
positive perturbation of $J$, there exists a  Hamiltonian
transformation $v$ of $M$ and $\delta \in (0,r)$ such that $\tilde
\ell_t=v\circ \ell_t$  defines
a local HSLAG fibration in 
$(M,\omega, J)$ for $t\in B(0,\delta)$.
\end{theointro}

The definition of \emph{positive perturbations} of $J_0$ shall be
given at \S\ref{sec:positive}. 
 Although we need this technical assumption for the proof
of Theorem~\ref{theo:B}, we conjecture that every generic almost complex
structure is positive. By definition, the set of positive almost
complex structures is open, once an appropriate topology is introduced on
$\AC_\omega$. In Theorem~\ref{theo:pospertintro} and
Theorem~\ref{theo:pospert}, we manage to prove that if $\ell_0:M\to L$ is rigid
and stable, then the space of positive almost complex structure form a non-empty
open set of $\AC$ with $J_0$ in its closure. To avoid technical
aspects of the statement, we state the result as follows:
\begin{theointro}
\label{theo:pospertintro}
  Let $(M,\omega,J_0)$ be a Kähler manifold and $\ell:L\to M$ be a rigid
  and stable HSLAG, where $M$ and $L$ are closed.

Then, the open set of positive deformations is not empty and has $J_0$
in its closure. More precisely,  there
exists a smooth path of almost complex structures $\tilde
J_s\in\AC_\omega$ defined for $s\in [0,\epsilon)$ with $\epsilon >0$,
  such that $\tilde J_0= J_0$ and $J_s$ is positive for all $s>0$.
\end{theointro}

In particular, Theorem~\ref{theo:pospertintro} shows that the
statement in Theorem~\ref{theo:B}, with the additional assumption of
stability of $\ell_0$, is not empty. This applies to the case of toric
HSLAG fibrations coming from toric Kähler manifolds given by Theorem~\ref{theo:torichslag}.

\begin{rmk}
\label{rmk:pity}
Considering $G^o_{\ell}$-invariant HSLAG toric fibration is not
really restrictive. Indeed, we shall prove in
Theorem~\ref{theo:equivexist} that  given a rigid
HSLAG embedding $\ell:L\to M$ into a Kähler manifold $(M,\omega,J_0)$, where $L$ is a real torus, it is always
possible to find a $G_\ell^o$-invariant HSLAG local fibration
$\ell_t:L\to M$ such that $\ell_0=\ell$, under some mild
assumptions. In particular, these assumptions apply in the context 
of toric Kähler manifolds (cf. Proposition \ref{prop:equivtoric}), which
constitute our main class of examples and applications. \end{rmk}

\begin{rmk}
The Hamiltonian transformation $v$ provided by Theorem \ref{theo:B}
may not be close to the identity as it comes from an auxiliary finite
dimensional minimization problem in the spirit of Theorem~\ref{theo:A}.
Some evidence of this fact are illustrated with a particular example (cf. Example~\ref{example:jump}).

   Theorem~\ref{theo:B} applies to the case of $\CP^n$
  with its Fubini-Study metric. The result only deals with positive
  pertubation of the metric, but  we expect it to hold under softer
  genericity assumptions.

More generally, we would expect that the standard singular HSLAG fibration by Lagrangian tori of 
 $\CP^n$  can be globally deformed for a
generic choice of K\"ahler metric close to the Fubini-Study metric on $\CP^n$. 
\end{rmk}

We gather our results in the case of toric manifolds in the following
corollary:
\begin{corintro}
  Let $(M,\omega,J_0)$ be a closed toric Kähler manifolds endowed with the Guillemin
  metric and $\mu:M\to P$ the corresponding moment map. Let $t_0$ be a
  point in the interior of the moment polytope $P$. Then 
  \begin{enumerate}
  \item for every compatible almost complex structure $J$
    sufficiently close to $J_0$, there exists a Hamiltonian
    transformation $v$ of $(M,\omega)$ such that $v(\mu^{-1}(t_0))$ is
    a HSLAG submanifold of $(M,\omega,J)$.
\item The space of positive deformations $J$ of the almost complex
  structure of $(M,\omega,J_0)$ with respect to $\mu^{-1}(t_0)$ is a
  non empty open set with $J_0$ in its closure. Is $J$ is positive and
  sufficiently close to $J_0$, there exists a Hamiltonian
  transformation $v$ with the property that each submanifold
  $v(\mu^{-1}(t))$ is HSLAG in $(M,\omega,J)$ for $t$ sufficiently close to $t_0$ in $P$.
  \end{enumerate}
\end{corintro}

\section{Basic theory of HSLAG}
In the rest of this paper, $M$ and $L$ are always a closed (i.e. compact without boundary) manifolds, unless specified otherwise. In addition $\dim M =2\dim L$, $\omega$ is a given symplectic form on $M$  and $\ell:L\to M$ is a Lagrangian embedding, that is an embedding such that $\ell^*\omega=0$.
\subsection{The Euler-Lagrange equation of HSLAG}
HSLAG manifolds are defined by a variational problem. The
corresponding Euler-Lagrange equation is easily recovered as we shall
explain now. Let $H$ be the mean curvature vector
field along $\ell:L\to M$ defined  using the metric~$g_J$. The \emph{Maslov
form} is the $1$-form $\alpha_H\in \Gamma(\Lambda^1L)$  defined by
$$
\alpha_H=\ell ^*\omega(H,\cdot).
$$
Let $f_t\in \Ham_\omega$ be
a family of Hamiltonian transformations such that $f_0=\id|_M$. Then
$$
V=\left . \frac{d}{dt}f_t\right |_{t=0}
$$
is a Hamiltonian vector field on $(M,\omega)$. 
Let $v:M\to \RR$ be a Hamiltonian function for $V$ (i.e. such that
$dv=\iota_V\omega$). The family $f_t$ induces a Hamiltonian deformation of the map $\ell:L\to M$
defined by  $\ell_t=f_t\circ\ell$. The resctriction of the vector field $V$ along
$\ell:L\to M$ is precisely the infinitesimal variation of the family $\ell_t$ at
$t=0$. The usual first variation formula of the volume gives
\begin{align*}
\left . \frac d{dt}\vol(\ell_t,g_J)\right |_{t=0} &=- \int_L\ip{H,V}\vol^{g_L}\\
&=-\int_L\ip{\alpha_H,\alpha_V}\vol^{g_L} \\
&=-\int_L\ip{\alpha_H,\ell ^*dv}  \vol^{g_L} \\
&=-\int_L\ip{d^*\alpha_H,\ell^*v}  \vol^{g_L} 
\end{align*}
where $g_L=\ell^*g_J$.

A critical point is obviously given by the equation
$$d^*\alpha_H=0,$$
where $d^*$ is the operator defined on the space of differential forms
on $L$, using the metric $g_L$.

 The second variation of the volume is obtained by
differentiating the $t$-dependent quantity $d^*\alpha_H$ along the
deformation $\ell_t$. It turns out that the second variation is of the form 
$$
\Box \ell^*v
$$
where $\Box$ is an elliptic
operator of order $4$ acting functions on $L$.
More precisely, we have the following formula~(cf. \cite{JLS11}), which holds
for any Lagrangian embedding (not necessarily stationary) in a
K\"ahler manifold 
\begin{equation}
\label{eq:box}
\Box u =\Delta^2u + d^*\alpha_{\Ric^\perp(J\nabla
  u)}-2d^*\alpha_{B(JH,\nabla u)} - JH\cdot JH\cdot u  .
\end{equation}
Here $\Delta$ is the Riemannian Laplacian on $(L,g_L)$, $B$ is the second fundamental form of $\ell:L\to M$ and
$\Ric^\perp$ is defined by $\Ric(x,y)=\ip{\Ric^\perp(x),y}$ for $x,y$
normal vectors to $L$. If the almost complex structure is not
integrable, the formula is more complicated, but the leading term
$\Delta^2$ remains unchanged.

 We introduce the Kernel of this operator
$$
\cK_L=\ker\Box.
$$

\begin{rmk}
\label{rmk:slag}
  If  $\ell:L\to M$ is a minimal submanifold (i.e. $H=0$), the
  operator $\Box$ is simpler. In particular, we have
$$
\ip{\Box u,u}=\int_L |\Delta u|^2 - \Ric(J\nabla u,J\nabla u) \vol^{g_L}
$$
and it follows that $\cK_L=\RR$ for any K\"ahler manifold of non
positive Ricci curvature. These conditions are fullfield in the case of
\begin{itemize}
\item \emph{special Lagrangians} in Calabi-Yau manifolds,
\item  standard Lagrangian tori
in the flat torus $\CC^2/\Gamma$.
\end{itemize}
\end{rmk}

Generally, $\cK_L$ is not reduced to $\RR$ and $\Box$ need not be selfadjoint.
However, we have the following lemma:
\begin{lemma}
\label{lemma:sa}
Given  a K\"ahler manifold $(M,\omega,J)$ and a Lagrangian embedding $\ell:L\to M$,
the operator $\Box$ on $L$ can be written 
$$
\Box v = D v - JH\cdot JH\cdot v
$$
where $D$ is selfajoint.

In particular, if $\ell:L\to M$ is HSLAG, the vector field $JH$ is
divergence free on $L$ and $\Box$ is selfadjoint.
\end{lemma}

\begin{proof}
Using $L^2$ inner product on $L$
  with  $g_L$ the induced metric, we see that
$\ip{\Delta^2v,w}=\ip{v,\Delta^2 w}$. Using the fact that
  $\alpha_{J\nabla w} = dw$ on $L$, we find 
$$
\ip{d^*\alpha_{\Ric^\perp(J\nabla
    v)},w}=\ip{\alpha_{\Ric^\perp(J\nabla v)},\alpha_{J\nabla
    w}}=\int_L \Ric(J\nabla v,J\nabla w) \vol^{g_L}
$$
Since the Ricci tensor is symmetric, we deduce that
$$
\ip{d^*\alpha_{\Ric^\perp(J\nabla
    v)},w} =
\ip{v,d^*\alpha_{\Ric^\perp(J\nabla
    w)}}
$$
Similarly, we compute
$$
\ip{d^*\alpha_{B(JH,\nabla v)},w } = \ip {\alpha_{B(JH,\nabla v)},
  \alpha_{J\nabla w}} = \ip{\alpha_{B(JH,\nabla v)},J\nabla v}.
$$
The tensor $S(u,v,w)= \ip{JB(u,v),w}$ is symmetric.
We deduce that
$$
\ip{d^*\alpha_{B(JH,\nabla v)},w } = \ip{v,d^*\alpha_{B(JH,\nabla w)} }.
$$
This proves the first statement of the lemma.

 Integration by part gives
$$\ip{-JH\cdot JH\cdot v,w}=\ip {JH\cdot v,JH\cdot w} +\int vw JH\cdot\vol^{g_L}.$$
If the embedding is HSLAG, then $d^*\alpha_H=0$. But this equation is
equivalent to the fact that $\mathrm{div} (JH)=0$ and in turns
${JH}\cdot\vol^{g_L}=\mathrm{div}(JH)\vol^{g_L}=0$. It follows that the
operator $v\mapsto -JH\cdot JH\cdot v$ is also selfadjoint.
\end{proof}

\subsection{Hamiltonian isometries, rigidity and stability}
Suppose that $(M,\omega,J_0)$ is a K\"ahler manifold and 
 and $\ell:L\to M$ is a HSLAG embedding.
Let $G$ be the
group of Hamiltonian isometries of the K\"ahler manifold. Any one
parameter subgroup $f_t:M\to M$ of $G$ gives rise to a deformation
$\ell_t:=f_t\circ\ell$ of $\ell$. However
$$
g_{L,t}=\ell_t^*g_{J_0}=\ell^*f_t^*g_{J_0} =\ell^*g_{J_0} = g_L
$$
since $f_t$ is an isometry. In particular the quantity $d^*\alpha_H$
is independent of $t$, therefore $\Box \ell^*v=0$,
for a Hamiltonian function $v:M\to\RR$ given by the variation~$f_t$.

The space of Hamiltonian functions induced by  $1$-parameter families
of Hamiltonian isometries of $(M,\omega,J_0)$ is the space of Killing
potentials denoted $\cK_{M}$.
Our observation shows that
 there is a canonical \emph{restriction map}
 \begin{align*}
\ell^*:   \cK_M &\to \cK_L\\
v &\mapsto v\circ \ell
 \end{align*}
where $\cK_L$ is the kernel of the operator $\Box$.
 \begin{dfn}
\label{dfn:rigidstable}
For a K\"ahler manifold $(M,\omega,J_0)$,
   a Hamiltonian stationary Lagrangian embedding $\ell:L\to M$ is called
   \emph{rigid} if the associated restriction map $\cK_M\to\cK_L$ is surjective.

If in addition, the operator $\Box$ is positive on a complement of the
kernel $\cK_L$, we says that $\ell:L\to M$ is \emph{stable}.
 \end{dfn}
According to Definition~\ref{dfn:rigidstable}, a Hamiltonian
Lagrangian embedding is rigid if, and only if, infinitesimal
Hamiltonian stationary deformations of the Lagrangian can only come
from of (globally defined) infinitesimal Hamiltonian
isometry of the Kähler manifold $(M,\omega,J_0)$.

 \begin{examples}$\qquad$
   \begin{enumerate}
\item   If $\cK_L=\RR$, the Lagrangian embedding $\ell:L\to M$ must be
   rigid.
\item
In particular all examples given by Remark~\ref{rmk:slag} (SLAG in
Calabi-Yau manifolds,
Lagrangian tori in the flat torus, HSLAG with $H=0$ in nonpositive
Ricci-curvature K\"ahler manifolds) are rigid.
\item 
It is also known that the Clifford torus in $\CP^n$ is rigid and stable
\cite{O90,O93}.
   \end{enumerate}
 \end{examples}

\section{Examples of Lagrangians toric fibrations}
\label{sec:toricfib}
Fibrations of Calabi-Yau manifolds  by special Lagrangian (SLAG) tori play a central role in mirror
symmetry. Unfortunately, these fibrations seem to be scarce
in nature. 
Hamiltonian stationary Lagrangians (HSLAG) are a natural generalization of
SLAG and they make sense on every (almost) K\"ahler manifold. 
We then turn to the question of existence of HSLAG fibrations for
K\"ahler manifolds.

A large pool of examples of such fibrations arise from toric K\"ahler
geometry. 
A toric K\"ahler manifold is endowed with a Hamiltonian isometric toric action.
Generic orbits of the torus action are Lagrangian tori. Since the
metric is invariant under the torus action, it follows that the mean
curvature of the orbits is also invariant and, in turn, that the Maslov form
$\alpha_H$ must be parallel. In particular $d^*\alpha_H=0$ and we have
the following lemma:
\begin{lemma}
\label{lemma:toric}
The fibration by Lagrangian tori of a toric K\"ahler
 manifold has the property that each smooth fiber
 (i.e. corresponding to an interior point of the polytope) is Hamiltonian
 stationary with respect to the toric metric.
\end{lemma}

%

\subsection{K\"ahler reductions}
Let $(\tilde{M},\tilde{\omega})$ be a symplectic manifold and $G$ a connected compact subgroup of Hamiltonnian diffeomorphisms. Assume that $\mu : \tilde{M} \ra \g^*$ is a $G$--equivariant momentum map and that $0\in \g^*$ is a regular value of $\mu$. Then $N=\mu^{-1}(0)$ is a $G$--invariant submanifold and $G$--orbits are coisotropic, so the quotient $\pi :\tilde{M}\ra M=\mu^{-1}(0)/G$  inherits of a symplectic structure $\omega$ defined as $\pi^*\omega =\tilde{\omega}_{|_{TN}}$. The symplectic orbifold obtained this way is called the \emph{symplectic reduction} of $(\tilde{M},\tilde{\omega})$.

Whenever there is a $G$--invariant compatible K\"ahler metric $\tilde{g}$ on $(\tilde{M},\tilde{\omega})$ then the whole K\"ahler structure descends on $M$ making the quotient map $$\pi : \mu^{-1}(0) \longrightarrow M $$ a Riemannian submersion. The resulting structure $(M,\omega, g)$ is called the \emph{K\"ahler reduction} of $(\tilde{M},\tilde{\omega},\tilde{g})$.

It would be interesting to see if the statonnary properties (being HSLAG, stable, rigid..) of a $G$--invariant Lagrangian lying in $\mu^{-1}(0) \subset \tilde{M}$ are preserved under this operation. We prove below it is the case for toric manifold.  

%

\begin{lemma}\label{lemKRHsLag} $\tilde{L}$ is a $G$--invariant Lagrangian of $(\tilde{M},\tilde{\omega})$ lying in $\mu^{-1}(c)$ if and only if $\pi(\tilde{L}) =L$ is a Lagrangian in $(M,\omega)$. Suppose that $\tilde{g}$ is a $G$--invariant compatible K\"ahler metric on $(\tilde{M},\tilde{\omega})$ such that there exists a positive constant $\kappa$ depending only on $(\tilde{M},\tilde{\omega},\tilde{g})$ and $G$ such that $\mbox{vol}(\pi^{-1}(L))= \kappa \,\mbox{vol}(L)$ for every Lagrangian $L$ in $M$. Then $L$ is HSLAG whenever $\tilde{L}$ is and is stable whenever $\tilde{L}$ is.  
\end{lemma}
\begin{proof} Put $N= \mu^{-1}(c)$.
The first statement follows the fact that for every $p\in N$, the kernel of $\omega_p$ in $T_pN$ (i.e those $v \in T_pN$ such that $\omega_p(v, w)=0$ for all $w \in T_pN)$) coincides with the tangent space of the orbit of $G$.

Given a function $f\in C^{\infty}(M)$, take any $\tilde{f} \in C^{\infty}(\tilde{M})$ which is a $G$--invariant extension of the function $\pi^*f \in C^{\infty}(N)^G$. Let $\tilde{X}$ be the Hamiltonian vector field associated to $\tilde{f}$. Observe that $\tilde{X}$ is tangent to $N$ since for each $a\in\g$ $$\langle d\mu(\tilde{X}), a\rangle = -\tilde{\omega}(X_a,\tilde{X}) = d\tilde{f} (X_a) =0.$$ 

Hence, for any Lagrangian $L \subset M$ and $\tilde{L}=\pi^{-1}(L)$ the variation $\tilde{\phi}_t(\tilde{L})$ induced by $t\tilde{X}$, stays in $N$ and $\tilde{\phi}_t(\tilde{L}) = \pi^{-1}(\phi_t(L))$. Thus, $\mbox{vol}(\tilde{\phi}_t(\tilde{L}))= \kappa\mbox{vol}(\phi_t(L))$. So that $L$ is HSLAG if $\tilde{L}$ is.    

Moreover, for any Lagrangian $L \subset M$, the $G$--invariant submanifold $\tilde{L}=\pi^{-1}(L)$ is Lagrangian in $(\tilde{M}, \tilde{\omega})$ via the composition of inclusions $\tilde{\iota}: \tilde{L} \hookrightarrow N \hookrightarrow\tilde{M}$ and $$\tilde{\iota}^*\tilde{\omega}(\tilde{X},\cdot) = \pi^*(\iota^*\omega(X_f, \cdot))$$ where $df=-\omega(X_f,\cdot)$. Hence, $\kappa \Box(f_{|_L}) = \Box(\pi^*(f_{|_L}))$. 
\end{proof}

\begin{cor}\label{coroRIGID} Assume, in addition to the hypothesis of Lemma~\ref{lemKRHsLag}, that $\tilde{L}$ is rigid then, for any function $f\in  C^{\infty}(L)$ such that $\Box f =0$, there exists a Killing potential $\tilde{f} \in C^{\infty}(\tilde{M})$ such that $\tilde{f}_{|_{\tilde{L}}}= \pi^*f$.
\end{cor}

\subsection{Case of $(\CP^n,\omega_{FS})$ and K\"ahler toric manifolds}\label{sectKReduction} 

Oh observed that each Lagrangian torus of $\CC^{n+1}$ of the form 
\begin{equation}\label{torus} T_{r_0,\dots, r_n} = \{ z\in  \CC^{n+1} \, | \, |z_i|= r_i\}\end{equation} is a critical point of the volume under Hamiltonian deformations (HSLAG). He proved also that it is rigid and stable by direct computations. One way to state his result is the following 

\begin{prop}[Oh, \cite{O93}] \label{propOh}
Let $\phi_t \in \mbox{Ham}(\CC^{n+1}, \omega_{FS} )$ be a path of Hamiltonian diffeomorphisms with $\pi_0=Id$ and $\tilde{X} = (\frac{d}{dt}\phi_t)_{t=0} \in \mathfrak{ham}(\CC^{n+1}, \omega_{FS})$. Let $\tilde{f}$ be a smooth Hamiltonian function for $\tilde{X}$. Then, $L=T_{r_0,\dots, r_n}$ is a HSLAG and $$\frac{d^2}{dt^2}\mbox{vol}(\phi_t(L)) \geq 0$$ Moreover, $\frac{d^2}{dt^2}\mbox{vol}(\phi_t(L)) =0 $ if and only if $\tilde{f}_{|_L}$ lies in the span of $$\{ \sin \theta_i, \, \cos \theta_i,\, \sin (\theta_i-\theta_j),\, \cos (\theta_i-\theta_j) \}_{i,j=0}^n.$$  
\end{prop}

\begin{rmks}
The hypothesis of Oh's result can be replaced by ``given a function $f \in C^{\infty}(L)$, or equivalently an exact $1$--form $\alpha_V =\iota^*\omega(V,\cdot)\in \Omega^1(L)$,...''.
\end{rmks}

By Delzant--Lerman--Tolman theory, see \cite{Delzant88, LermanTolman97}, any compact toric symplectic orbifold $(M,\omega,T)$ is the symplectic reduction of $(\CC^d,\omega_{std})$ with respect to a subtorus $G \subset \bT^{d}$. Via this construction the manifold inherits of a K\"ahler metric $\omega_G$, called the Guillemin metric~\cite{Guillemin94}. 

We recall the Delzant construction which is determined by the rational labelled polytope $(\pol, \nu,\Lambda)$ associated to $(M,\omega,T)$. Denote the Lie algebra $\kt= \mbox{Lie } T$ and the momentum map $x : M \ra \pol\subset \kt^*$. We use the convenient convention that $\nu=\{\nu_1,\dots, \nu_d\}$ is a set of vectors in $\kt$ so that if $F_1,\dots, F_d$ are the codimension $1$ faces ({\it the facets}) then $\nu_k$ is normal to $F_k$ and inward to $\pol$. The lattice $\Lambda$ defines the torus as $T=\kt/\Lambda$. Being {\it rational}\footnote{To recover the original convention introduced by Lerman and Tolman in the rational case, take $m_k\in \ZZ$ such that $\frac{1}{m_k}\nu_k$ is primitive in $\Lambda$ so $(\pol, m_1,\dots m_d,\Lambda)$ is a rational labelled polytope.} means 
$\nu\subset \Lambda$. From the rational pair $(\pol, \nu)$, the torus $G$ is determined by its Lie algebra $\g= \ker \q$ where $\q : \RR^d \ra \kt$ is $$\q(x):=\sum_{i=1}^dx_i\nu_i.$$ Hence, the compactness of $P$ implies that $\q$ is surjective and that $T= \RR^d/G$. 
  
The defining affine functions of $(\pol, \nu)$ are the $\ell_1,\dots, \ell_d\in \mbox{Aff}(\kt^*, \RR)$ such that $\pol= \{ x\in\kt^*\,|\,\ell_k(x)\geq0\}$ and $d\l_k=\nu_k$. Denote the inclusion $\iota : \g \hookrightarrow \RR^d$ and $\mu_o : \CC^d \ra (\RR^d)^*$ the (homogenous of degree $2$) momentum map of the action of $ \bT^{d}$ on $\CC^d$ so that $\iota^*\mu_o$ is a momentum map for the action of $G$ on $\CC^d$. One side of the Delzant--Lerman--Tolman correspondence states that $(M,\omega,T)$ is $T$--equivariently symplectomorphic to the symplectic reduction of $(\CC^d,\omega_{std})$ at the level $c=\iota^*(\ell_1(0),\dots,\ell_d(0)) \in \g^*$. In the following, we identify $(M,\omega,T)$ with this reduction. \\  
Note that since $\mu_o (z) = \frac{1}{2}(|z_0|^2,\dots,|z_n|^2)$ the defining equations of $N=(\iota^*\mu_o)^{-1}(c)$ involve only the square radii $r_i^2=|z_i|^2$ and, thus, $N$ is foliated by tori (of various dimension between $\dim G$ and $d$) of the form \eqref{torus}. Moreover, for each  $x\in \mathring{\pol}$ the interior of $\pol$, $L_x =\mu^{-1}(x) \subset M$ is a Lagrangian torus such that $\pi^{-1}(L_x) =\tilde{L}_x$ is a $d$--dimensional torus of the form \eqref{torus}. Finally, observe that if $\tilde{f}$ is a Killing potential of $(\CC^d,\omega_{std})$ which is $G$--invariant on {\it  some} $d$--dimensional torus of the form \eqref{torus} then $\tilde{f}$ is $G$--invariant on {\it any} $d$--dimensional torus of the form \eqref{torus}.

\begin{lemma}\label{vol_cst_MOMENT}  Let $G$ be a subtorus of $ \bT^{d}$ and $\iota : G\hookrightarrow \bT^{d}$ the inclusion. The volume, with respect to the standard flat metric of $\CC^d$, of the orbits of $G$ is constant on the regular level set of the momentum map $\iota^*\mu_o : \CC^d \rightarrow \mathfrak{g}^*$.     
\end{lemma}

\begin{proof}
 Let $g$ be the standard metric on $\CC^d$. We have $$g=\sum_{i=1}^d \frac{d\mu_i^2}{2\mu_i} +  2\mu_i d\theta_i^2$$ where $\mu_i=\frac{1}{2}|z_i|^2$ and $\theta_i$ is the angle coordinates. Hence, for $z\in \CC^d$ such that $G\cdot z$ is of full dimension, the action identifies $G\cdot z$ with $G$ and the metric induced on the orbit is $\iota^*h$ where $h= 2\mu_i d\theta_i^2$. Clearly $\iota^*h$ only depends on $\iota^*(\mu_o (z))$.
\end{proof}

This last lemma implies that the present situation fulfill the hypothesis of Lemma~\ref{lemKRHsLag} which we combine with  Corollary~\ref{coroRIGID} and Proposition~\ref{propOh}, to get the following Proposition.

\begin{prop}\label{propGMetrigid} Any $L_x\subset M$ Lagrangian torus obtained as level set of the momentum map $\mu: M \ra \kt^*$ is HSLAG, stable and rigid for the Guillemin metric $\omega_G$. 
\end{prop}
 \begin{rmk}
  Lemma~\ref{vol_cst_MOMENT} applies more generally to {\it toric rigid metrics} in the sense of \cite{H2FII}. Therefore there is a bigger class of metric on which Proposition~\ref{propGMetrigid} extends.
 \end{rmk}

The complex projective space with its Fubini-Study metric $(\CP^n,\omega_{FS})$ is a case of that last Proposition. Indeed, one way to see $\CP^n$ is as a K\"ahler reduction of $\CC^{n+1}$ with respect to the diagonal Hamiltonian action of $S^1$. This coicides with the Hodge fibration $\pi : S^{2n+1} \ra \CP^n$ and fits with the toric structure of $(\CP^n,\omega_{FS})$ obtained by quotient of the isometric toric action of $\bT^{n+1}$ on $(\CC^{n+1},\omega_{std})$. We denote this quotient $T= \bT^{n+1}/S^1$, $\kt =\mbox{Lie } T$ and the momentum map $\mu : \CP^n \ra \kt^*$. For any $L_x\subset \CP^n$ Lagrangian torus obtained as level set of the moment map, $\pi^{-1}(L_x)$ is a torus of the form~\eqref{torus}. It may be more convenient to take the (more natural) momentum map $$[Z_0:\dots : Z_n] \mapsto \left(\frac{|Z_0|^2}{\sum_{i=0}^n |Z_i|^2},\cdots, \frac{|Z_n|^2}{\sum_{i=0}^n |Z_i|^2}\right)\in \RR^{n+1}$$ so that the $r_i$'s in~\eqref{torus} are just the $|Z_i|$'s.\\
The $S^1$ invariant functions on $T_{r_0,\dots,r_n}$ in the kernel of $\Box$ are $\{\sin(\theta_i-\theta_j), \cos(\theta_i-\theta_j)\}_{i,j =0}^n$ so Proposition~\ref{propGMetrigid} holds in this case and reads as follow. 

\begin{prop}[\cite{Ono2007}] Any $L_x\subset \CP^n$ Lagrangian torus obtained as level set of the moment map $\mu: \CP^n \ra \kt^*$ is HSLAG, stable and rigid for the Fubini-Study metric. 
\end{prop}

\section{Deformation theory}
\label{sec:defo}
In this section, we are assuming that $\ell :L\to M$ is a Hamiltonian
stationary Lagrangian
embedding, where $(M,\omega,J_0)$ is a K\"ahler manifold. For each
almost complex structure $ J\in \AC_\omega$ sufficiently close
to $J_0$, we would like to find a Hamiltonian deformation $\tilde\ell :
L\to M$ of the Lagrangian embedding
$\ell : L\to M$ which is Hamiltonian stationary with respect to
$(M,\omega, J)$. It turns out that this problem can not be solved directly
by the implicit function theorem since the equations are generally overdetermined. 

\subsection{Diffeomorphisms of the source space}
The group of diffeomorphisms  $\Diff(L)$ acts on  Lagrangian
embeddings $\ell: L\to M$ by composition on the right. Such an
infinite dimensional group action gives a huge group 
of symmetries which preserves the equation for HSLAG embeddings. This indeterminacy
of the equations is
easily removed by seeking Hamiltonian deformations of the image
$$\cL=\ell(L)$$
as a
Lagrangian submanifold of $M$ instead. This boils down to consider the space of Lagrangian embeddings  upto reparametrizations.

\subsection{Group of isometries}
From now on, we are assuming that  $\ell:L\to M$ is a HSLAG embedding with respect to
$(M,\omega,J_0)$ and that it is rigid.

The group $G$ of Hamiltonian isometries of $(M,\omega, J_0)$ has a
corresponding space of Hamiltonian potentials  $\cK_M\subset
C^\infty(M)$. Their restriction to $L$ is  $\ell^*(\cK_M)\subset
C^\infty(L)$ and agrees with $\cK_L=\ker \Box_{\ell,J_0}$ by the
rigidity assumption.
An issue when trying to apply directly the implicit function
theorem is that the linearization of the HSLAG equations given by the
operator $\Box_{\ell,J_0}$ is generally neither injective nor surjective.

\subsection{Lagrangian neighborhood theorem}
A sufficiently small tubular
neighborhood $\cV$ of $\cL=\ell(L)$ in $M$ is symplectomorphic to a neigborhood $\cU$ of the zero
section $\ell_0:L\to T^*L$ endowed with its canonical symplectic
form. In addition $\ell$ is identified to $\ell_0$ via the
symplectomorphism. 
Small Lagrangian  deformations of $\cL$ are given by graphs of
sections $\alpha $ of $T^*L\to L$, where $\alpha$ is a sufficiently
small closed $1$-form on $L$. 
Furthermore, Hamiltonian deformations are given by exact
$1$-forms. 
Thus, every  smooth function $f$ on $L$, defines a
Lagrangian submanifold of $T^*L$ which is the graph of $df:L\to T^*L$
and this graph is a Hamiltonian deformation of the zero
section $\ell_0:L\to T^*L$. 
If $f$ is sufficiently small (in $C^1$-norm), using the
symplectomorphism between the tubular neighborhoods $\cU$ and $\cV$,
each section $df$ defines a Lagrangian embedding $\ell_f:L\to M$ which
is a Hamiltonian deformation of $\ell : L\to M$.

\subsection{Vector bundles}
The image of the map
$\ell^*:\cK_M\to
C^\infty(L)$
is $\cK_L=\ker\Box_{\ell,J_0}$ by rigidity, but it may not be injective. After passing to a subspace
$\cK_M^o\subset \cK_M$, we obtain an isomorphism
$$\ell^*:\cK^o_M\to \cK_L.$$

For any embedding $\tilde \ell:L\to M$
sufficiently close to $\ell:L\to M$, the map $\tilde \ell^*:\cK^o_M\to
C^\infty(L)$ remains injective. Using this observation, we are going
to introduce vector bundles of finite rank over the space of functions.

However, the space of smooth functions $C^\infty (L)$ is not suited to  apply
the implicit function theorem. Instead, we shall work with H\"older
spaces $C^{k,\eta}(L)$, where $0<\eta<1$ is the Hölder parameter and
$k$ is the number of derivatives accounted for.  For any function $f\in C^{4,\eta}(L)$
sufficiently small, the map $\ell_f^*:\cK_M^o\to C^{3,\eta}(L)$
remains injective. Its image is a finite dimensional vector space denoted
$$\fK_f\subset C^{3,\eta}(L).$$
 The spaces $\fK_f$ are the fibers of a
smooth vector bundle $\fK$ over a neighborhood of the origin in $C^{4,\eta}(L)$.

Let $\cH'$ be the orthogonal complement of $\cK_L=\fK_0$ in
$C^{3,\eta}(L)$, where the $L^2$ inner product  is induced by $\ell$
and $J_0$. For any $f\in  C^{4,\eta}(L)$ sufficiently small, we have a
splitting of vector bundles
$$
C^{3,\eta}(L)=\fK_f\oplus \cH',
$$
and the projection on the first factor parallel to $\cH'$ is denoted
$$
\pi_f :C^{3,\eta}(L)\to \fK_f.
$$
Similarly, we will need to consider the $L^2$-orthogonal complement
$\cH$ of $\cK_L$ in $C^{4,\eta}(L)$ which gives the splitting
\begin{equation}
  \label{eq:splith}
  C^{4,\eta}(L)=\cK_L\oplus \cH.
\end{equation}

\subsection{Implicit function theorem}
We introduce the map 
\begin{equation}
\label{eq:Psi}
\Psi:C^{4,\eta} (L)\times C^{4,\eta}(L) \times \AC^{2,\eta}_\omega\to C^{0,\eta}(L)  
\end{equation}
defined as follows: let $(f,k,J)$ be an element of $C^{4,\eta}(L)\times C^{4,\eta}(L)\times \AC^{2,\eta}_\omega$. The function 
$f $  admits a 
decomposition $f=f_L +  h$ where $f_L\in \cK_L$ and $h\in \cH$
according to the splitting \eqref{eq:splith}.
For $f$ and $k$ sufficiently small, we may use the Lagrangian
embedding $\ell_{k+h}:L\to M$.  
We define 
$$
\Psi(f,k, J) = d^*\alpha_H+ \pi_{h+k}(f) 
$$
where $d^*\alpha_H$ is computed with respect to the Lagrangian
embedding $\ell_{h+k}:L\to M$ and the almost complex structure $J$.

The differential of $\Psi$ at $(0,0,J_0)$ is given by
$$
\left . \frac{\del \Psi}{\del f}\right |_{(0,0,J_0)} \!\!\!\!\!\!\cdot
\dot f =   \dot f_L + \Box_{\ell,J_0} \dot h
$$
where we used the  decomposition $\dot f =\dot f_L+ \dot h \in \cK_L \oplus \cH$.
This operator is clearly an isomorphism.

By the implicit function function theorem, we deduce the following
proposition
\begin{prop}
  \label{prop:hslag}
There are open
neighborhoods $U, V$ of  $0$ in $C^{4,\eta}(L)$ and a $G$-invariant
open neigborhood $W$ of the almost complex structure $J_0$ in
$\AC_\omega^{2,\eta}$
  together
with a smooth map 
$$
\phi:V\times W\to U
$$
such that
$$
\Psi(\phi(k,J),k,J)=0
$$
for all $(k,J)\in V\times W$. Furthermore  $\phi(k,J)$ is the only solution $f\in U$ of the equation $\Psi(f,k,J)=0$ where $(k,J)\in V\times W$.  
\end{prop}

By definition, a solution of the equation $\Psi(f,k,J)=0$ provides a
Lagrangian embedding $\ell_{h+k}:L\to M$ satisfying the equation
\begin{equation} \label{eqOBSTmaslov}
d^*\alpha_H\in \fK_{h+k}. 
\end{equation}

Thus,  our problem of finding a HSLAG embedding is solved
up to a finite dimensional obstruction. The solution of this type are
called \emph{relatively HSLAG} embeddings.
\begin{dfn}
  A Lagrangian embedding $\ell:L\to M$ into an almost K\"ahler
  manifold $(M,\omega, J)$ such that $d^*\alpha_H$ is the restriction
  of a Hamiltonian potential in $\cK_M$ is called \emph{a relatively HSLAG
  embedding}.
  \end{dfn}

More information on the regularity of $(k,J)\in V\times W$ provide more
regularity on $f=\phi(k,J)$ by standard bootstrapping argument for
elliptic equations. In particular, we have the following lemma:
\begin{lemma}
  \label{lemma:bootstrap}
  If $(k,J)\in V\times W$ are smooth, then $f=\phi(k,J)$ is smooth and so is
  $d^*\alpha_H$ for the corresponding relatively HSLAG embedding.
\end{lemma}

It is worth pointing out that
many cases are dealt with using the following corollary:
\begin{cor}
Let $\ell:L\to M$ be a rigid HSLAG embedding  into $(M,\omega,J_0)$ with
$\cK_L=\RR$. Then for all compatible almost complex structure $J$
sufficiently close to $J_0$ in $C^{2,\eta}$-norm, there exists a
Hamiltonian deformations $\tilde l:L\to M$ of $\ell$ which is a HSLAG
embedding with respect to $(M,\omega,J)$.
\end{cor}
\begin{proof}
For $J$ sufficiently close to $J_0$, we use the decomposition
$\phi(0,J)=f_L+h$ given by the splitting~\eqref{eq:splith}. By
assumption $f_L$ must be a constant.
  The embedding $\ell_{h}:L\to M$ satisfies the equation $d^*\alpha _H = c$,
  for some constant $c\in \RR$,
  with respect to the almost K\"ahler structure $(M,\omega,  J)$. Since $d^*\alpha_H$ is $L^2$-orthogonal to constants, we deduce
  that $c=0$.
\end{proof}

\subsection{Residual isometry group action and killing the
  obstruction}
\label{sec:residual}
The next step is to look for solutions of the equation $\Psi=0$ such
that the finite dimensional obstruction \eqref{eqOBSTmaslov} vanishes.
The residual
group action of $G$ on $W$ is the key to achieve this goal. Indeed, we have
a {\it modified volume functional}
$$\voltilde: W \to \RR$$
defined by
\begin{equation}
  \label{eq:volfun}
  \voltilde(J) = \vol(L,\ell^*_{h}g_{ J}).
\end{equation}
where $\phi(0,J)=f =f_L+h\in \cK_L\oplus \cH$ is given by
Proposition~\ref{prop:hslag}.

Notice that $\voltilde$ is a perturbation of the volume
functional  
$$\vol:W\to \RR$$ 
of $\ell:L\to M$
 defined by $\vol(J)=\vol(L,\ell^*g_J)$. Nevertheless, $\voltilde$ and
 $\vol$ need not to agree generally.

Each $G$-orbit of almost complex structure in $W$, is compact, since
$G$ is. Thus the modified volume
functional $\voltilde$ restricted to a $G$-orbit admits critical points, for
instance a minimum. Let $J$ be such a point in a given orbit.
Then we have the following result.
\begin{prop}
  \label{prop:critical}
Let $J$ be a smooth almost complex structure, which is a critical
point of the modified volume functional $\voltilde$ given by~\eqref{eq:volfun},
restricted to a $G$-orbit of $W$.
Then, the Lagrangian embedding $\ell_h:L\to M$ deduced from
$\phi(0,J)=f_L+h$  via Proposition~\ref{prop:hslag} is HSLAG.
\end{prop}

As a direct consequence, we obtain a proof of one of our main results:
\begin{proof}[Proof of Theorem~\ref{theo:A}]
We have a map $W\to U$ given by $J\mapsto \phi(0,J)$. For each $J$ the
decomposition $\phi(0,J)=f_L+h \in\cK_L\oplus \cH$ provides a
Lagrangian embedding $\ell_h$, which is a Hamiltonian deformation of
$\ell$.  An easy exercice of symplectic geometry shows that one can define a smooth map $\psi$ on
$W$ such that $\psi(J)$ is a Hamiltonian
transformation of $(M,\omega)$ with the property that
$\psi(J)\circ\ell = \ell_h$ and $\psi(J_0)=\id$ as
in Theorem~\ref{theo:A}. We give an outline of the argument to keep
this paper self-contained. The function $h$ is a priori defined on
$L$. However a tubular neighborhood of $\cL=\ell(L)$ is identified
with a neighborhood $\cV$ of the $0$-section of the contangent bundle
$T^*L\to L$. The function $h$ can be understood as a function on $\cV$
by pull back. We fix a suitable smooth compactly supported cut-off
function $\varphi$ on $\cV$ equal to $1$ in a neighborhood of $\cL$. Then
$\varphi h$ makes sense as a globally defined function on $M$. The
corresponding Hamiltonian vector field $X_{\varphi h}$ is well defined on
$(M,\omega)$. By integrating upto time $1$, the flow of the vector
field defines  Hamiltonian transformation $\psi(J)$ of $(M,\omega)$, with
regularity a $C^{4,\eta}$. If $J$ is sufficiently close to $J_0$, the
function $h$ is very close to $0$ in $C^{3,\eta}$-norm. In particular
we have $\psi(J)\circ\ell = \ell_h$, by construction.

By assumption $J$ is smooth here, hence by Lemma~\ref{lemma:bootstrap} 
so is $h$. Thus $\psi(J)$ is also a smooth Hamiltonian transformation
so that we can avoid the complication of introducing a group of Hamiltonian
transformations with suitable Hölder topology.

For the second part of the theorem, given $J\in W$, it suffices to find $u\in G$ such that $u\cdot J$ is a
critical point of the functional $\voltilde$ (such $u$ exists by
compactness of $G$). Then $u\cdot J$ 
satisfies the claim of Theorem~\ref{theo:A} thanks to Proposition~\ref{prop:critical}.  
\end{proof}

 The rest of this section is devoted to the proof of the proposition.  
\begin{proof}[Proof of Proposition \ref{prop:critical}]
  Let $f=\phi(0,J)$. Notice that by Lemma \ref{lemma:bootstrap}, the
  function $f$ is smooth and so is $h$. The smooth embedding, $\ell_{ h}:L\to M$, 
satisfies the equation $d^*\alpha_H=\psi$ for some $\psi\in\fK_{
  h}$, which is  also smooth.
By definition, $\psi= v \circ \ell_{ h}$ for some function $v\in
\cK_M$, by rigidity.

Let $u_t\in G$ be a $1$-parameter subgroup of $G$ such that
$u_0=id$ and $\frac{du_t(x)}{dt}|_{t=0}=X_v(x)$. In other words, the
tangent vector to the $1$-parameter subgroup is the Hamiltonian vector
field $X_v$ associated to $v$.

We consider the orbit $\tilde J_t=u_t^*J$  under the
$1$-parameter subgroup action. Since $J$ is a critical point of the modified
volume functional on its $G$-orbit, we have
$$
\left . \frac{d}{dt}\right |_{t=0}\voltilde (\tilde J_t) = 0.
$$
Looking more closely at the modified volume functional, this means that the
volume is critical for the $1$-parameter family of Lagrangian
embeddings $\ell_{\tilde h_t}:L\to M$ with respect to the almost
complex structures $\tilde J_t$, where $\phi(0,\tilde J_t)=\tilde f_t =\tilde
f_{L,t}+\tilde h_t\in \cK_L\oplus \cH$. By Lemma~\ref{lemma:bootstrap},
the functions $\tilde f_t$, $\tilde h_t$ and $\tilde f_{L,t}$ are all
smooth. Furthermore, they depend smoothly on $t$.

Changing the point of view, this means that the volume of the
Lagrangian embeddings $\tilde \ell_t:=u_t\circ \ell_{\tilde h_t}:L\to M$  is critical with
respect to the fixed almost complex structure $J$.

Since the volume of a Lagrangian embedding depends only on the volume
of its image, we may work upto the action $\Diff(L)$. Hence, we may
assume after composing $\tilde \ell_t$ on the right by a suitable
family of diffeomorphism that they are given by $\ell_{\tilde h_t + k_t}:L\to M$, where $k_t$ is a smooth
family of functions on $L$ such that $\frac{\del k_t}{\del t}|_{t=0}=\psi$.

The solution of the relative HSLAG equations are invariant under the
action of $\Diff(L)$, therefore we have a one parameter family of
solutions of the equation
$$
\Psi(\tilde f_t, k_t,J)=0,
$$
with critical volume at $t=0$.

Differentiating at $t=0$ gives the identity
\begin{equation}
  \label{eq:diffpsi}
\frac{\del\Psi}{\del f}|_{(f,0,J)}\cdot \dot f =- \frac{\del\Psi}{\del k}|_{(f,0,J)}\cdot
\dot k
\end{equation}
where
$$
\dot f=\left . \frac {\del \tilde f_t}{\del t}\right |_{t=0} \quad \mbox{ and
} \quad \dot k=\left . \frac
     {\del k_t}{\del t}\right |_{t=0} =\psi.
     $$
The computation of the operator $\frac{\del\Psi}{\del k}$ is similar
to $\frac{\del\Psi}{\del h}$
and we find
$$
\left .\frac{\del\Psi}{\del k}\right |_{(f,0,J)}\!\!\!\!\!\cdot
\dot k = L_{f,J}\cdot \dot k +\Box_{\ell_h,J}\dot k.
$$
where  $L_{f,J}\cdot\dot k = \frac
{d}{dt}\pi_{h+k_t}(f)|_{t=0}$.
We choose $\epsilon_0>0$ sufficiently small, to be fixed afterward.
According to Lemma~\ref{lemma:piest}, upto passing to sufficiently small
open sets $U$ and $W$, we have the estimates
$$
\|L_{f,J}\cdot \dot k\|_{L^2}\leq \epsilon_0 \|\dot k\|_{L^2}
$$
and since $\dot k=\psi \in \fK_h$, by Lemma~\ref{lemma:boxest}
$$
\|\Box_{\ell_h,J}\dot k \|_{L^2}\leq \epsilon_0 \|\dot k\|_{L^2}.
$$

We deduce that the $L^2$-norm of the LHS of \eqref{eq:diffpsi} is
bounded by $2\epsilon_0\|\psi\|_{L^2}$. Applying
Lemma~\ref{lemma:unifest}, we obtain the estimate
$$
\|\dot f\|_{L^2}\leq 2C\epsilon_0 \|\psi\|_{L^2}.
$$
In particular we have an estimate
$$
\|\dot h\|_{L^2}\leq 2C\epsilon_0 \|\psi\|_{L^2}.
$$
If we choose $\epsilon_0 =\frac 1{4C}$, we have
$$
\|\dot h\|_{L^2} \leq \frac 12\|\psi\|_{L^2}.
$$
 By the first variation formula,
we have 
$$
\dot \vol = \ip {\alpha_H,d\dot h + d\psi}  = \ip{d^*\alpha_H,\dot h +  \psi}
= \ip{\psi,\dot h + \psi}=\|\psi\|^2 + \ip{\psi,\dot h}.
$$
By Cauchy-Schwartz inequality, it follows that
$$
\dot\vol
\geq \|\psi\|^2 - \|\psi\|\|\dot h\|\geq
\frac 12\|\psi\|^2.
$$
This is a contradiction unless $\psi=0$.
\end{proof}
Here are the technicals lemmata used in the proof of
Proposition~\ref{prop:hslag}. These results are standard and some
proofs may be omited.

The operator $\frac{\del\Psi}{\del h}|_{0,0,J_0}$ is an isomorphism.
 The first eigenvalue of small perturbations of this
operator remain uniformly bounded away from $0$ in the sense of the
following lemma:
\begin{lemma}
  \label{lemma:unifest}
  For all neighborhood $U$ and $W$ of Proposition \ref{prop:hslag} chosen sufficiently small, there exists a constant
  $C>0$ such that
   all $(f,J)\in U\times W$ and  $ F \in
  C^{4,\eta}(L)$ 
  $$
\|F\|_{L^2}\leq C \left \|\left .\frac{\del\Psi}{\del h}\right |_{(f,0,J)}\!\!\!\!\!\cdot
F\right \|_{L^2}.
$$
\end{lemma}

\begin{lemma}
    \label{lemma:piest}
  For all $\epsilon >0$ there are sufficiently small open sets $U$ and
  $W$  in Proposition \ref{prop:hslag} such that for all $(f,J)\in U\times W$ and all function $F\in \fK_h$ 
  $$
\|L_{f,J}F\|_{L^2}\leq \epsilon \|F\|_{L^2}.
  $$
\end{lemma}
\begin{proof}
  The proof is easily done by contradiction. If the lemma is not true,
  there exists $\epsilon >0$ and a family $(f_j,J_j)\in U\times W$ converging toward
  $(0,J_0)$ and $F_j\in \fK_{h_j}$ such that
  $\|F_j\|_{L_2}=1$ while  $\|L_{f_j,J_j}F_j\|_{L^2}>\epsilon$.

  Using the fact that $\fK$ is a finite rank bundle over $U\times W$
  with a smoothly varying $L^2$-inner product, we  may assume, after
  passing to a subsequence, that $F_j$ converge to $F\in \fK_0=\fK_L$
  with the property that $\|F\|_{L ^2}=1$ and $L_{0,J_0}F \neq 0$. But
  this is impossible since $\pi_h(0)$ is identically $0$, so that
  $L_{0,J_0}\equiv 0$.
  \end{proof}

\begin{lemma}
    \label{lemma:boxest}
  For all $\epsilon >0$ there are sufficiently small open sets $U$ and
  $W$, as in  Proposition \ref{prop:hslag}, such that for all function $F\in \fK_h$, we have
  $$
\|\Box_{\ell_h,J}F\|\leq \epsilon \|F\|_{L^2}.
  $$
\end{lemma}

\section{The volume functional}
This section gathers some results and observations on the volume
functional that will be needed to show that the space of positive
perturbations, introduced at \S\ref{sec:positive} in not empty, under mild assumption, in particular for the proof of Theorem~\ref{theo:pospertintro} and
Theorem~\ref{theo:pospert}.

\subsection{Main technical result}
We start with a Kähler manifold $(M,\omega,J_0)$. Let $G$ be its group of
Hamiltonian isometries. We consider a rigid HSLAG
embedding $\ell:L\to M$ as in Theorem~\ref{theo:A}. 

Let $G_\ell$ be the
subgroup of isometries of $G$ preserving the image of $\ell:L\to M$. In
other words $u\in G_\ell$ if and only if $u\circ\ell(L)=\ell(L)$. We
denote by $G_\ell^o$ the identity component of $G_\ell$.

We consider the subspace of compactible almost complex
structures which differ from $J_0$ by a diffeomorphism of $M$. The
connected component of $J_0$ in this space is denoted
$\cJ_\omega \subset\AC_\omega$.

The standard volume functional here refers to the map
$$
\vol_J : G/G^o_\ell\to \RR
$$
defined by $\vol_J([u])= \vol(L,\ell^*g_{u\cdot J})$. 


The variational formulas for the volume functional are much easier to
carry out assuming the base complex structure is integrable but the
result is likely to hold for almost K\"ahler metrics. Of course, it
holds for any almost K\"ahler metric $(J,\omega)$ in a sufficiently
small neighborhood of the K\"ahler metric $(J_0,\omega)$.  
Then main technical result  of this section is the following theorem:

\begin{theo}
\label{prop:evelinedeg}
  There exists a smooth $1$-parameter family of complex structures
  $J_s\in W$ defined for $s\geq 0$ such that 
  \begin{enumerate}
  \item   $J_0$ is our given complex structure
    \item $\ell:L\to M$ is HSLAG with respect to $J_s$ for all $s\geq
      0$
\item $\vol_{J_s}:G/G^o_\ell\to \RR$ admits a non-degenerate local
  minimum at $u=\id$ for all $s>0$.
  \end{enumerate}
\end{theo}

\subsection{Some variational formulae}
First, we have to carry out some variational formulas. In order to do that we identify a neighborhood of $L$ with a neighborhood of the zero section in $T^*L$ and a set of coordinates $x_1,\dots, x_n$ on $L$ is completed on $T^*L$ to give Darboux coordinates $\omega= \sum_{i=1}^n dx_i\wedge dy_i$. In these coordinates, $dy_i$ vanishes on $TL$, so the volume form  $\vform_{J}$ on $L$ is $$\vform_{J} = \sqrt{|g_L|}dx_1\wedge \cdots \wedge dx_n $$  where $g_L = g_{ij} dx_i\otimes dx_j$ is the restriction of the metric $g$ on $L$, $|g_L|$ the determinant of $(g_{ij})$. The variation of $\Phi$ along a path $\{\mbox{id}\}\times J_t$ starting at $J$ is then given by \begin{equation}\label{varVol0} \frac{d}{dt}_{t=0}\vol_{J_t}(\mbox{id})= \int_L \mbox{tr}(g_L^{-1}\dot{g_L})\vform_{J} = \int_L \mbox{tr}(g_L^{-1}\dot{g})\vform_{J}.\end{equation}
More generally, for a given $u\in G$, $$ d_J\Phi (\dot{J})\,(u) = \frac{d}{dt}_{t=0}\vol_{J_t}(u)=  \int_{u(L)} {\tr} (g_{u(L)}^{-1}\dot{g})\vform_{J}.$$

For $J\in \cJ_{\omega}$, the tangent $T_J\cJ_{\omega}$ is a subspace of endormorphism of $TM$ and has a natural complex structure $\JJ$ given by  $$\JJ A = J\circ A$$ for $A\in T_J\cJ_{\omega}$. Denoting the orbit of $J$ by under $\mbox{Symp}_0$ by $\cG.J=\mbox{Symp}_0(\omega)\cdot J$, one can prove, see~\cite{PG}, that \begin{equation}\label{correspTANGENTJ} T_J\cJ_{\omega} = \{-\cL_Z J\,|\, Z\in \mathfrak{symp}(\omega) + J \mathfrak{ham}(\omega) \}\,  \simeq \,T_J\cG.J + \JJ T_J\cG.J.\end{equation}
The last equality uses $\cL_{JZ}J =J\cL_{Z}J$ which holds thanks to the integrability of $J$. Note that the right hand side of~\eqref{correspTANGENTJ} is not a direct sum in general.

Let $J_t$ be a path in $\cJ_{\omega}$ defined for small $t$ such that $g_t = \omega(J_t\cdot,\cdot)$ be the corresponding variation of Riemannian metrics, so that $g_0=g$. 

For $\phi\in C^{\infty}(M)$, let $X_\phi \in \mathfrak{ham}(\omega)$ be the corresponding Hamiltonian vector field that is $d\phi = -\omega(X_\phi,\cdot)$. We denote by $D$ the Levi-Civita connection and $d^c\phi = -d\phi\circ J$. Recall that for a $1$--form $\alpha \in \Omega^1(M)$, the Levi-Civita is defined as $D\alpha \,(\!X,Y\!) = X.\alpha(Y) - \alpha (D_XY)$ and the Hessian of $\phi$ is $Dd\phi$, a symmetric tensor since $D$ has no torsion. The $J$--invariant and anti-invariant parts of $D\alpha$ are $$D^{\pm}\!\alpha \,(\!X,Y\!) = \frac{1}{2}(D\alpha\,(\!X,Y\!) \pm D\alpha\,(\!JX,JY\!)).$$  
\begin{lemma}\label{varJ} Let $\phi\in C^{\infty}(M)$,
\begin{itemize}
\item[a)] if $\dot{J} = -\cL_{X_\phi} J$, then $\dot{g} = -2D^-d^c\phi$.
\item[b)] if $\dot{J} = -J\cL_{X_\phi} J$, then $\dot{g} = 2D^-d\phi$.
\end{itemize}
\end{lemma}
\begin{proof} Since $\omega$ does not vary along the path $(\omega, g_t,J_t)$ and that $\cL_{X_\phi} \omega=0$, in the case a), we have, using \cite[Lemma 1.20.2]{PG}, that  
\begin{equation} \dot{g} =-\omega(\dot{J},) = \omega(\cL_{X_\phi} J(\cdot),\cdot) =-2 D^-(X_\phi^\flat) =-2 D^-d^c\phi
\end{equation} since $g(X_\phi,\cdot) =d^c\phi$. In case b), we have \begin{equation*} \begin{split}\dot{g} =-\omega(\dot{J},) = \omega(J\cL_{X_\phi} J(\cdot),\cdot) &=-g( \cL_{X_\phi}J\cdot,\cdot)\\
&=  -(\cL_{X_\phi} \omega) + \cL_{X_\phi}g(J\cdot,\cdot) = \cL_{X_\phi}g(J\cdot,\cdot). \end{split} 
\end{equation*}  Since $(g,J)$ is K\"ahler, $DJ=0$ and so, for $Z,Y \in\Gamma(TM)$, we have
\begin{equation}\begin{split}
\dot{g}(Y,Z) &= \cL_{X_\phi}g(JY,Z) = g(D_{JY}X_\phi,Z) + g(JY,D_{Z}X_\phi) \\
&= (JY).g(X_\phi,Z) - g(X_\phi, D_{JY}Z) + (Z).g(JY,X_\phi) - g(D_{Z}JY,X_\phi) \\
&= (JY).d^c\phi (Z) - d^c\phi(D_{JY}Z) + Z.d^c\phi (JY) - d^c\phi(D_{Z}(JY))\\
&= -(JY).d\phi (JZ) + d\phi(D_{JY}JZ) + Z.d\phi (Y) - d\phi(D_{Z}(Y))\\
&= -Dd\phi \,(\!JY,JZ\!) + Dd\phi \,(\!Z,Y\!)\\
&= -Dd\phi \,(\!JY,JZ\!) + Dd\phi \,(\!Y,Z\!) = 2D^-d\phi(Y,Z)
\end{split}
\end{equation}\end{proof}

Consequently, taking a variation $\dot{J} = -\cL_{X_\phi} J - J\cL_{X_\psi} J$ we have  
\begin{equation}\label{imagedPHI}
 d_J\Phi(\dot{J})(u) =-2\int_{u(L)} {\tr} \left(g_{u(L)}^{-1}(D^-d^c\phi - D^-d\psi)\right) \vform_J.
 \end{equation} 
~\\

 \begin{lemma}\label{varJ} Let $\phi\in C^{\infty}(M)$ and $X, Y\in T_pM$ then
\begin{itemize}
\item[a)] $D^-d^c\phi(X,Y) = -D^-d\phi(JX,Y)$,
\item[b)] $2D^+d\phi(X,Y) =dd^c \phi (X,JY)$,
\item[c)] $2(D^-d\phi(JX,X) + D^-d\phi(X,X))=dd^c \phi (X,JX)$.
\end{itemize}
\end{lemma}
\begin{proof}
 For a), recall that $d^c\phi =-d\phi\circ J$. We use that $DJ=0$ as follow 
 \begin{equation}\begin{split}
 2D^-d^c\phi(X,Y) &= Dd^c\phi(X,Y)-Dd^c\phi(JX,JY)\\
 & = X.d^c\phi(Y)- d^c\phi(D_XY) - (JX).d^c\phi(JY) +d^c\phi(D_{JX}JY)\\
 & =-X.d\phi(JY)+ d\phi(D_XJY) - (JX).d\phi(Y) +d\phi(D_{JX}Y)\\
 &= -Dd\phi(X,JY)-Dd\phi(JX,Y)\\
  &= Dd\phi(J^2X,JY)-Dd\phi(JX,Y)\\
  & =-2D^-d\phi(JX,Y)
 \end{split}
\end{equation}
 For b), note that $Dd\phi$ is the Hessian, thus symmetric, hence 
  \begin{equation}\begin{split}
 2D^+d^c\phi(X,Y) &= Dd\phi(X,Y) +Dd\phi(JX,JY)\\
 & = X.d\phi(Y)- d\phi(D_XY) + (JY).d\phi(JX) -d\phi(D_{JY}JX).
 \end{split}
\end{equation} Using that $D$ has no torsion, we see it coincides with 
 \begin{equation}\begin{split}
 dd^c\phi(X,JY) &= X.d^c\phi (JY) -(JY).d^c\phi(X) -d^c\phi(D_XJY - D_{JY}X)\\
 & = X.d\phi(Y) + (JY).d\phi(JX)- d\phi(D_XY) -d\phi(D_{JY}JX).
 \end{split}
\end{equation}
For c), using again that $Dd\phi$ is symmetric we have
 \begin{equation}\begin{split}
 2(D^-d\phi(JX,X) + D^-d\phi(X,X))&= Dd\phi(JX,X) + Dd\phi(X,JX) \\
 & \qquad\qquad\;\; + Dd\phi(X,X)-Dd\phi(JX,JX) \\
 & = Dd\phi(X,X+JX) + Dd\phi(JX,X- JX)\\
 & = 2D^+d\phi(X,X-JX)\\
 \end{split}
\end{equation} Now, using formula b) just above, we have $2D^+d\phi(X,X-JX) = dd^c\phi(X, J(X-JX)) = dd^c\phi(X, JX) + dd^c\phi(X,X)= dd^c\phi(X, JX)$. \end{proof}

\subsection{Particular variations}
Let $\varphi$ be a real smooth function on $M$. We consider a family of almost complex structures  $J_s\in
\cJ_\omega$, defined for $s\in \RR$ sufficiently close to zero and
such that
\begin{equation}\label{varJalongVARPHI}
\frac {\del}{\del s}J_s = -\cL_{X_\varphi} J_s + J_s\cL_{X_\varphi} J_s
\end{equation}
for all $s$ sufficiently small,
where $X_\varphi$ is the Hamiltonian vector field deduced from~$\varphi$.

We have the following variation formula for the volume:
\begin{prop}\label{propFORMULAvarVOL} Let $\varphi_s$ be a smooth family of real smooth functions on $M$ for small $s\in \RR$ and consider the corresponding 
 of $J_s\in \cJ_\omega$ satisfying \eqref{varJalongVARPHI}. Then
$$
\frac {\del}{\del s}\vol_{J_s}(u) =\int_L \left ((\Delta^{g_s}\varphi) \circ u
\circ \ell\right )\; \vol((u\circ \ell)^*g_{J_s}).
$$
\end{prop}
\begin{proof} 
Using the first and last formulas of Lemma~\ref{varJ}, we get that 
 \begin{equation}\label{1stVARvolGJ} d_J\Phi(\dot{J})(u) = -2\int_{u(L)} {\tr} (g_{u(L)}^{-1} dd^c\phi(\cdot,J\cdot)) \vform_J = -2 \int_{u(L)} \Delta^g\phi \; \vform_J \end{equation} where $ \Delta^g\phi$ is the Laplacian of $\phi$ on $M$ with respect to the K\"ahler metric $g$. Indeed, taking $\{v_i\}_{i=1}^n$ an orthonormal basis of $T_pL$, the trace $${\tr} (g_{u(L)}^{-1} dd^c\phi(\cdot,J\cdot)) = \sum_{i=1}^n dd^c\phi(v_i,Jv_i)$$ is the symplectic trace of $dd^c\phi$ which turns out to be the Laplacian of $\phi$ at $p$ up to a sign, see for e.g.~\cite[p.33]{PG}. Observe that to obtain \eqref{1stVARvolGJ} we didn't use the fact that $G$ fixes $J$ so the formula holds at any point of $\cJ_{\omega}$. 
\end{proof}

The variation of $J_s$ depends on first derivatives of $\phi$. If
$\phi$ vanishes upto order $1$ along the image of $\ell:L\to M$, the
almost complex structures $J_s$ are independent of $s$ along $\ell$ as
well. The metric $g_s$ deduced from $\omega$ and $g_s$ must also be
independent of $s$ along $\ell$. Is $\varphi$ vanishes to a higher
order, say upto order $2$ along $\ell$, the mean curvature along
$\ell$ will also be independent of $s$. Since the metric, the mean
curvature are independent of $s$ we have the following lemma
\begin{lemma}
  Let $\varphi:M\to \RR$ be a smooth function vanishing upto order $2$
  along the image of $\ell:L\to M$, and $J_s$ be a family of almost
  complex structures defined from $\varphi$ as above.
Then $\ell: L\to M$ is HSLAG with respect to $J_s$ for all $s$
sufficiently small.
\end{lemma}

For $s=0$, since $u$ acts by isometry on $g_{J_0}$ we
deduce from Proposition~\ref{propFORMULAvarVOL} that
$$
\frac {\del}{\del s}\vol_{J_s}(u)|_{s=0} =\int_L \left ((\Delta^{g_0}\varphi) \circ u
\circ \ell\right )\; \vol(g_L).
$$

\begin{lemma}
  Let $(N^k,g)$ be a smooth compact Riemannian manifold of dimension $k$ and $p\in N$. Let $r:N\to \RR$ be the Riemannian distance to
  $p$ in $N$. The function $r^4$ is a smooth function in the neigborhood of $p\in N$ with the
  property that
$$
\Delta (r^4) = -4r^2(k+2) + {\bf o}(r^3),
$$
where $\Delta$ is the Laplacian associated to the metric $g$.
\end{lemma}
\begin{proof} It is well-known that $r^2$ is smooth, so does $r^4$. Taking normal coordinates $(r,\theta)$ centered at $p$ we know the metric tend up to order $1$ to the Euclidean metric at $p$. Then $|\nabla r|^2 \ra 1$ and $\Delta r^2 \ra -2 k$ when $r \ra 0$. It suffices to use the formula $\Delta r^4 = 2r^2\Delta r^2 -2g(\nabla r^2, \nabla r^2) = 2r^2(\Delta r^2 - 4g(\nabla r, \nabla r))$ to get the result.  
\end{proof}

All the ingredients above can be used to give a proof of the
main result of this section.
\begin{proof}[Proof of Theorem~\ref{prop:evelinedeg}]
Returning to our setup, still identifying a neigborhood of $\ell(L)$ in $M$ with a neigborhood $\cV$ of the $0$--section in $T^*L$ and considering the K\"ahler metric induced by $g$ on $\cV$. Denote the natural projection $T^*L \rightarrow L$.  We consider the function $\varphi :\cV\ra \RR$ defined by $\varphi(\alpha)= \frac{1}{-4(n+2)}d^g(\pi(\alpha),\alpha)^4$ where $d^g$ is the distance induced by $g$. Then, when restricted on a fiber $T_p^*L\cap \cV$, $\varphi$ is the $4$-th power of the distance function to $p$ up to a constant factor and satisfies the last lemma. If $(r_p,\theta_p)$ are normal coordinates of $T_p^*L\cap \cV$ centered at $p$ then $\varphi(p,r_p,\theta_p) = \frac{r_p^4}{-4(n+2)}$. The lemma above (see the proof to be convinced it works as well) gives then $$\Delta^g\varphi= r_p^2+\cO(r_p^3).$$
Let $u_t\in G$  be a one parameter subgroup of $G$ such that
$u_0=\id$. Since $\varphi$ vanishes upto order $3$ along $\ell:L\to M$
we have $\Delta^{g_0}\varphi =0$ upto order $1$ along $\ell$, where
$g_0$ is the Riemannian metric deduced from $J_0$ and $\omega$. Therefore 
$$\frac\del{\del t }\left ((\Delta^{g_0}\varphi)\circ u_t\circ
\ell\right )|_{t=0} =0$$
on $L$. In particular
$$\frac {\del}{\del t}\frac {\del}{\del
  s}\vol_{J_s}(u_t)|_{(s,t)=0}=0.$$

The second order $t$-derivative is given by
$$
\frac {\del^2}{\del t^2}\frac {\del}{\del
  s}\vol_{J_s}(u_t)|_{(s,t)=0} = \frac 12 \int_L Q(X,X) \vol(g_L)
$$
where $Q$ is the Hessian quadratic form of $\Delta^{g_0}\varphi$ and $X$ is the
Hamiltonian vector field tangent to $u_t$ at $t=0$ along $\ell$.
If $\Delta^{g_0}\varphi =r_p^2+\cO(r_p^3)$, we deduce that $Q$ is definite
positive in directions transverse to $\ell$. If $u_t$ is transverse to
$G_\ell^0$, then $X$ is not everywhere tangent to $\ell$. Hence we have proved that
$$\frac {\del^2}{\del t^2}\frac {\del}{\del s}\vol_{J_s}(u_t)|_{(s,t)=0}>0$$
unless $u_t$ is tangent to $G_\ell^o$ at $t=0$.

We deduce that if $u_t$ is transverse to $G_\ell^0$ at $t=0$
$$
\vol_{J_s}(u_t) = c + b_st^2 + o(t^2)
$$
where the constant $b_s$ is strictly positive for $s>0$. This
completes the proof of Theorem~\ref{prop:evelinedeg}.
\end{proof}

\subsection{Other properties of the volume functional}
Our variationnal formulas can be used to show that the standard volume
functional has Morse properties for generic almost complex
structures. Unfortunately we are interested in the modified volume
functional~\eqref{eq:volfun} in this paper, and this is why we relied on
Theorem~\ref{prop:evelinedeg} instead. For the interested reader, we
state the following result, which is in the spirit of the proof of existence 
of Morse function in the finite dimensional setting, although
it is not used in the rest of the paper:
\begin{prop}\label{propMORSE} Given a compact K\"ahler manifold
$(M,\omega, J_0)$ and a compact Lagrangian $\ell: L\hookrightarrow
M$. The map $\Phi :\cJ_\omega \longrightarrow
C^{\infty}(\quot{G}{G_\ell})$, defined
as \begin{equation}\label{mapPHI} \Phi(J):= \vol_J \end{equation} is
a submersion in a neigborhood of $J_0$.    
\end{prop}

\begin{proof} The map $u\mapsto
u(L)$ is injective on $\quot{G}{G_\ell}$ so there exists a point
$p\in L$ such that there is a neighborhood $V$ of $id \in G$
satisfying the condition: \begin{equation}\label{condV} \forall u\in
V, \; u(p) \in L \mbox{ if and only if } u\in G_\ell. \end{equation}  

 The neighborhood $V$ depends on $p$ but since it is an open
 condition on $p$ we can choose a neighborhood $U$ of $p$ in $L$ such
 that condition \eqref{condV} holds. Then the orbit map $\psi(q,
 \gamma)= \gamma(q)$ induces a smooth foliation of the image
 $\cW:=\psi(U\times V)$, in $M$. The leafs of $W$ are $u(L)\cap
 W\simeq U$ for $u\in G/G_\ell$. Actually, we have a
 diffeomorphism \begin{equation}\label{diffeoLOC} \tau: \cW
 \stackrel{\sim}{\longrightarrow} U\times V_\ell  \end{equation}
 where $V_\ell\subset G/G_\ell$ denotes the image of $V$ via the
 quotient map $G\rightarrow G/G_\ell$. 
 
 We can easily pushforward any bump function $\psi_U$ from $U$ to the
 whole $\cW$ via the action of $G$ and any function $f$ on $V_\ell$
 defines a $G_\ell$--invariant function on $V$. The pull-back
 $\tau^*(\psi_U\times f)$ on $\cW$ may be extended to $M$ so that it
 integrates to $0$. Taking the Green function of this extension to be
 the variation $\dot{J}$ as above, we get that \eqref{1stVARvolGJ}
 becomes $$d_J\Phi(\dot{J})(u) = f(u)$$ for all $u\in V_\ell$. From
 which we conclude that $d_J\Phi$ is surjective on
 $C^{\infty}(V_\ell)$. 
\end{proof}

\section{Deformation theory for local HSLAG toric fibrations}
\subsection{Obstacles to overcome}
So far, we only considered the case of a single HSLAG embedding
$\ell:L\to M$. We would like
to extend the theory to the case of a local HSLAG toric fibration
$\ell_t:L\to M$  as defined in \S\ref{sec:fibdef}.
There are several issues for extending
Theorem~\ref{theo:A} to a Lagrangian fibration:
\begin{enumerate}
\item The fibration becomes singular at the boundary of the
  polytope in the case of the standard fibration by Lagrangian tori of
  a toric Kähler manifold. In the case of a SLAG
  fibration of a K3 surface, certain fibers have several irreducible
  components. Such issues related involve complicated analytical problems that we shall not
  tackle at this paper. This is why we restrict our attention to local
  fibrations by smooth Lagrangian tori as in \S\ref{sec:fibdef}.  
\item Proposition~\ref{prop:critical} involves the choice a local minimum of the volume functional
  seen as a function on $G$. One cannot make a consistent choice a minimum
  for a family of tori, unless the volume has some special
  properties. We shall prove that for a generic choice of metric the
  volume functional is non degenerate, which allows to get around this issue.
\end{enumerate}

\begin{example}
\label{example:jump}
The fact that the Hamiltonian transformation $v$ of
Theorem~\ref{theo:B}  is not
necessarily small can be readily observed in an example. 

We consider
the unit $2$-dimensional sphere in $\RR^3$ with a given axis going from the north to the south pole and its standard complex structure obtained by rotating each tangent plane by an angle $\pi/2$. 

Every embedded circle $C$ is
automatically Lagrangian in such a low dimensional case. Furthermore,
the HSLAG property is equivalent to the fact that the circle has
constant curvature. There is a standard fibration $\ell_t:S^1\to S^2$, for $t\in (-1,1)$ of the sphere by circles of constant curvature, known as the parallels, obtained by rotating the sphere about its axis
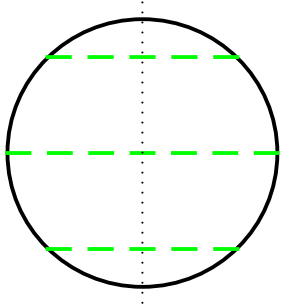
\begin{center}
\begin{pspicture}[showgrid=false](-2,-2)(2,2)
\psrotate(0,0){90}{
\psscalebox {1.8}{
\pscircle(0,0){1}
\psset{linecolor=green, linestyle=dashed}
\psline(! 45 cos  45 sin) (! -45   cos  -45  sin)
\psline(! 90 cos  90 sin) (! -90   cos  -90  sin)
\psline(! 135 cos  135 sin) (! -135   cos  -135  sin)
}
\psset{linecolor=black, linestyle=dotted}
\psline(-2,0) (2,0)}
\end{pspicture}
\end{center}

We pick an axis $(D)$
going through the shere center, wich is distinct from the given axis, and does not belong to the equator plane. We consider a
deformation of the sphere, with prescriped area, into an ellipsoïd of revolution, with axis
$(D)$. As shown by the picture below, there is an obvious fibration by
circle of constant curvature.
However this fibration is far from the original fibration which suggests that 
a jump must occur.
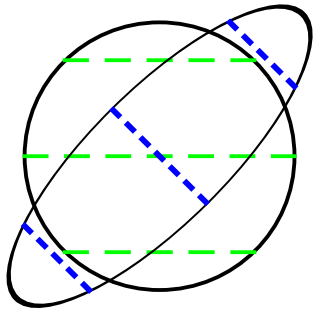
\begin{center}
\begin{pspicture}[showgrid=false](-2,-2)(2,2)
\psrotate(0,0){90}{\psscalebox {1.8}{
\pscircle(0,0){1}
\psset{linecolor=green, linestyle=dashed}
\psline(! 45 cos  45 sin) (! -45   cos  -45  sin)
\psline(! 90 cos  90 sin) (! -90   cos  -90  sin)
\psline(! 135 cos  135 sin) (! -135   cos  -135  sin)

\psset{linecolor=black, linestyle=solid}
\psrotate(0,0){-45}{\psscalebox {1.5 0.5}{
\pscircle(0,0){1}
\psset{linecolor=blue, linestyle=dashed}
\psline(! 45 cos  45 sin) (! -45   cos  -45  sin)
\psline(! 90 cos  90 sin) (! -90   cos  -90  sin)
\psline(! 135 cos  135 sin) (! -135   cos  -135  sin)
}
}}}
\end{pspicture}
\end{center}
However, this only a heuristic argument  since we do not have a proof that this HSLAG fibration is the only one on an ellipsoïd of revolution.

We now give a correct argument.
The equator $\ell_0:S^1\to S^2$ is a geodesic, hence its Maslov form $\alpha_0$ vanishes. On the other hand, the sign of the integral $\int_{S^1}\alpha_t$  of the Maslov form $\alpha_t$ of $\ell_t$ changes when $t$ goes through $0$. 

We consider a variation $J_s$ of the standard complex structure $J_0$ on the sphere $S^2$, compatible with the symplectic form. The Maslov form of $\ell_t$ now depends on $s$ as well, and we denote it by $\alpha_{t,s}$. By continuity, there exists $t_s$ for each sufficiently small deformation $J_s$ such that $\int_{S^1}\alpha_{t_s,s} =0$. If there is a Hamiltonian transformation $v_s$ of the sphere such that $v\circ\ell$ is a HSLAG fibration with respect to $J_s$, then $v\circ \ell_{t_s}$ has constant curvature with exact Maslov form and it must be a geodesic. 

Now, we pick a particular variation $J_s$,  where each $J_s$ induces the metric of an ellipsoïd of revolution with three distinct axis unless $s=0$. 
Such surfaces are known to have only three closed geodesics. Upto a rotation, we may assume that none of the three geodesics agree with the equator of the starting round sphere. This shows that $v_s$ is not close to the identity for $s$ close to $0$.
\end{example}

\subsection{Positive perturbations}
\label{sec:positive}
We start with a Kähler manifold $(M,\omega,J_0)$. Let $G$ be its group of
Hamiltonian isometries. We consider a rigid HSLAG
embedding $\ell:L\to M$ as in Theorem~\ref{theo:A}. 

Let $G_\ell$ be the
subgroup of isometries of $G$ preserving the image of $\ell:L\to M$. In
other words $u\in G_\ell$ if and only if $u\circ\ell(L)=\ell(L)$. We
denote by $G_\ell^o$ the identity component of $G_\ell$.

Using the notation of  \S\ref{sec:residual}, we consider almost
complex structures $J\in W$ sufficiently close to $J_0$ and the
corresponding relatively HSLAG embedding $\ell_h:L\to M$ on
$(M,\omega,J)$. 


The modified volume functional $\voltilde: W \to \RR$ is generally not
$G$-invariant. However $G_\ell$ leaves $\cL=\ell(L)$ invariant by
definition. It follows that the map $\Psi$ defined at \eqref{eq:Psi} is $G_\ell$-equivariant,
and so is the map $\phi$ defined in Proposition~\ref{prop:hslag}. In turn, the  modified volume functional
$\voltilde:W\to \RR$ is $G_\ell$-invariant.
Hence, the modified volume functional may be understood as a map
$$
\voltilde_J:G/G^o_\ell\to \RR
$$
defined by $\voltilde_J([u])=\voltilde(u\cdot J)$. By Proposition~\ref{prop:critical},  critical
points of this functional correspond to HSLAG embeddings. Such critical
points are generally not non-degenerate. For instance, the choice
$J=J_0$ provides a constant function $\voltilde_{J_0}$ since $J_0$ is
$G$-invariant and all critical points are degenerate.

\begin{dfn}
 An almost complex structure $J\in W$ is called 
a \emph{positive deformation of $J_0$} with respect to the rigid HSLAG
$\ell:L\to M$  if the corresponding functional
$\voltilde:G/G^o_\ell\to \RR$ admits a non degenerate local minimum. The
subset of $W$ that consists of positive deformations is denoted $W^+$.
\end{dfn}

We have the following obvious result, by stability of non-degenerate
local minimum:
\begin{lemma}
  The set of positive deformations $W^+$ of $J_0$ is an open subset of
  $W$, endowed with
  its $C^{2,\eta}$-topology.
\end{lemma}

The openness result does not insure that $W^+$ is non-empty. Although
we suspect that it is never empty, we shall prove it under some
reasonnable technical assumptions:
\begin{theo}
\label{theo:pospert}
  Let $(M,\omega,J_0)$ be a Kähler manifold and $\ell:L\to M$ be a rigid
  and stable HSLAG.

Then, the open set of positive deformations $W^+\subset W$ of $J_0$ is not
empty. Furthermore, there exists a smooth family of 
complex structure $J_s\in W$, defined for $s\geq 0$, such that $J_s\in W^+$ for
all $s>0$. In other words, $J_0$ is in the closure of~$W^+$. 
\end{theo}
Notice that Theorem~\ref{theo:pospertintro} is a less technical
restatement of the above theorem.
\begin{proof}
  
Let $J_s$ be a family of complex structures as in the above
proposition. 
\begin{lemma}
  Under the stability assumption of Theorem~\ref{theo:pospert}, the
  functional
$$\voltilde_{J_s}:G/G^o_\ell\to \RR$$
 admits non-degenerate local
  minimum at the identity for every $s>0$ sufficiently small.
\end{lemma}
\begin{proof}
The fact that $\voltilde$ admits a critical point at the identity is
clear, since $\ell:L\to M$ is a HSLAG in $(M,\omega,J_s)$.

The only thing to be proved is the fact that it is non-degenerate.
Let $u_t\in G$ be a one parameter subgroup of $G$ transverse to
$G_\ell$, with $u_0=\id$  and $\dot k\in \cK_M$ a corresponding Hamiltonian
function such that $\frac d{dt}u_t|_{t=0}=X_{\dot k}$.

Let $J_{s,t}=u_t^* J_s$ and $\ell_{t,s}= \ell_{h_{s,t}}:L\to M$ be the corresponding
solution, a relatively HSLAG embedding, provided by the implicit function theorem as in Proposition~\ref{prop:hslag}.

By definition $\ell_{s,0}=\ell$. We can switch our point of view using
$\tilde \ell_{s,t}=u_t\circ\ell_{s,t}$ with a fixed complex structure
$J_s$.
By the second variation formula, we have
$$
\frac {d^2}{dt^2}\vol(\tilde\ell_{s,t},g_{J_s})|_{t=0}=
\ip{\Box_{J_s}(\dot k+\dot h),\dot k +\dot h}.
$$
where $\dot h\in \cH$ is the projection of $\frac \del{\del k}\phi|_{(0,J_s)}\cdot \dot k$ on
$\cH$.
Expanding the above inner product and integrating by part
we obtain
$$
\ip{\Box_{J_s}\dot k,\dot k} + \ip{\Box_{J_s}\dot h,\dot h} + 2 \ip{\Box_{J_s}\dot k,\dot h}
$$
For $s>0$,  since $J_s$ is provided
by Theorem~\ref{prop:evelinedeg} and we have the non-degeneracy
property $(3)$. 
The second term is non
negative for $s$ sufficiently small by the stability assumption of
$\ell:L\to M$ into $(M,\omega,J_0)$. The
third term is controlled by the first two terms since $\Box_{J_s}\dot
k$ converges to $0$ as $s$ goes to $0$ and the eigenvalues of
$\Box_{J_s}$ are uniformly bounded from below in the direction $\dot
h$.
Hence the second variation of the volume must be positive for $s>0$
sufficiently small.
\end{proof}

In conclusion, the complex structures $J_s$ belong to $W^+$ for $s>0$
sufficiently small. This completes the proof of Theorem~\ref{theo:pospert}.
\end{proof}

\subsection{Invariant Lagrangian fibrations}
We would like to extend the deformation theory of \S\ref{sec:defo} to the case
of a fibration. For this purpose, we choose to restrict to the case of
$G^o_\ell$-invariant fibrations in the sense of Definition~\ref{dfn:equivariant}.
This technical assumption is not too demanding as it is satisfied by
examples provided by toric Kähler geometry. We state few observations on $G^o_\ell$-invariant fibrations in the following Proposition.

\begin{prop} 
\label{prop:equivtoric}Let $\ell_t : L \simeq M^{2n}$ be a Lagrangian fibration with $\ell_0=\ell$ and $t\in B(0,\epsilon)$ and such that $L$ is compact.
\begin{enumerate}
 \item If $G_\ell^o$ acts effectively on $\ell(L)$ then $G_\ell^o$ is a torus of dimension at most $n$. 
 \item If $\ell_t$ is a $G^o_\ell$-invariant fibration then $G_\ell^o$ is a torus of dimension at most $n$.
 \item If $G_\ell^o=\TT^n$ then there exists a $G^o_\ell$-invariant fibration $\tilde{\ell}_t :L \ra M$ such that $\tilde{\ell}_0=\ell$ is a neigborhood of $\ell(L)$. 
 \item If $(M,g,\omega,J)$ is a toric K\"ahler manifold with momentum map $\mu: M\ra P\subset \RR^n$, then $\{\mu^{-1}(p)\,|\, p\in\mathring{P} \}$ is a $\TT^n$--invariant fibration. 
\end{enumerate}
\end{prop}

\begin{proof} 
 The first affirmation follows the observation that the orbit of $G_\ell^o$ in $\ell(L)$ must be isotropic, thus a torus, thanks to the formulas
 \begin{equation}\label{eqHAM} d\omega(X_a,X_b) =-\omega([X_a,X_b],\cdot)=-\omega(X_{[a,b]},\cdot)
 \end{equation} where $X_a$ is the vector field induced on $M$ by $a\in \mbox{Lie }G_\ell^o$. The second affirmation is a consequence of the observation that in the case of $G^o_\ell$-invariant fibration we have $G^o_{\ell} \subset G^o_{\ell_t}$ and there is an open and dense subset of $t\in B(0,\epsilon)$ such that $G^o_{\ell_t}$ acts effectively on $\ell_t(L)$. For the third one, we consider the generic orbits of $G_\ell^o=\TT^n$, which must be Lagrangian by the formula \eqref{eqHAM} above. The fourth affirmation is obvious.
\end{proof}

\begin{theo}
\label{theo:equivexist}
  Let $(M,\omega,J_0)$ be a Kähler manifold and $\ell:L\to M$ a rigid HSLAG
  embedding, where $L$ is a real torus.

Assume that non-trivial harmonic forms on $L$ for the induced metric
do not vanish at any point. Then there exists a $G^o_\ell$-invariant
HSLAG fibration $\ell_t:L\to M$ such that $\ell_0=\ell$.
\end{theo}
%

\begin{proof}[Proof of Theorem~\ref{theo:equivexist}]
  The compact group $G_\ell$ preserves the image of $\ell:L\to M$. Thus
  $G_\ell$ has an induced action on $L$ by diffeomorphism. This action
  also induces a symplectic $G_\ell$-action on $T^*L$. The
  starting point of our setup to apply the implicit function Theorem (cf. \S\ref{sec:defo}) requires the choice of a
  symplectic diffeomorphism between a neighborhood of the image of
  $\ell$ in $M$ and a neighborhood of the zero-section in $T^*L$. This
  symplectomorphism can be chosen to be $G_\ell$-equivariant.

We have a Riemannian metric $g_L$ on $L$ induced by $g_{J_0}$. Since
$G_\ell$ acts isometrically on $(M,g_{J_0})$, the induced action on
$(L,g_L)$ is also isometric. In particular, $G_\ell$ acts on the space
$\cH^1(L,g_L)$ of harmonic $1$-forms of $(L,g_L)$. Since elements of
$G^o_\ell$ are homotopic to the identity in $\Diff(L)$, they act
trivially on the cohomology of $L$, hence on the space of harmonic $1$-forms $\cH^1(L,g_L)$.

One can construct a standard Lagrangian toric fibration
$$
\cH^1(L,g_L)\times L \to T^*L
$$
given by $(\alpha,x)\mapsto \alpha_x$. This construction is $G_\ell$
equivariant, by definition. Using the $G_\ell$-equivariant
indentification between a neighborhood of the
$0$-section of $T^*L$ and a neighborhood of the image of $\ell$, we
deduce a local lagrangian toric fibration
$$
\hat \ell : K \times L \to M
$$
where $K$ is a $G_\ell$-invariant neighborhood of the origin in
$\cH^1(L,g_L)$. For $K$ sufficiently small, this map is indeed an
embedding since by assumption, harmonic $1$-forms do not vanish at any point.
We use the notation $\hat\ell_\alpha = \hat\ell(\alpha,\cdot)$ in the sequel.

By definition $\hat \ell_0=\ell$ so it must be HSLAG. However
$\hat\ell_\alpha$ may not be HSLAG for $\alpha\in K$. By construction the fibration is
$G_\ell$-equivariant and the action induced by $G_\ell^o$ is trivial
on the parameter space $\alpha\in K$. 

Using a version of the implicit function theorem with parameter as in
Proposition~\ref{prop:hslag}, one
can perturb each map $\hat\ell_\alpha$ for $\alpha\in K$ by a Hamiltonian deformation,
provided $K$ is sufficiently small, in
order to get a relatively HSLAG Lagrangian embedding. More precisely,
there exists a smooth map
$$
\phi : K \to U,
$$
with the notations of Proposition~\ref{prop:hslag}, 
such that the lagrangian embedding $\ell_\alpha$ defined by the
$1$-form $\alpha + dh_\alpha$, where $h_\alpha = \phi(\alpha)$ is
relatively HSLAG.

 By uniqueness of
the solution of the IFT and the fact that $G$ acts by isometries on
$g_{J_0}$ we obtain that $\phi$ is a $G_\ell$ equivariant map.
In particular, the Lagrangian fibration $\ell_\alpha : L\to M$ is also
$G_\ell^o$-invariant. 

The invariance of the metric also implies that
the volume of $\ell_\alpha$ is invariant under the action of $G$. This
forces the equation $d^*\alpha_H=0$ by the first variation formula for
the volume. Therefore, each $\ell_\alpha$ must be HSLAG.
\end{proof}
\subsection{Fibrations and positive perturbations}
At this stage, all the tools necessary to handle the case of HSLAG
fibrations have been introduced.

Let  $\ell_t:L\to M$ be a
HSLAG toric fibrations into a Kähler manifold $(M,\omega,J_0)$, with
$G$ its the group of
Hamiltonian isometries. We are assuming that $\ell_t$ is
$G_{\ell_0}^o$-invariant.

We are assuming that $J\in W^+$ is a positive perturbation with respect
to $\ell=\ell_0$. By stability of non-degenerate minimum, we deduce
that $J$ is positive with respect to every $\ell_t$, for $t$
sufficiently small.

Provided $\ell_0$ is rigid, using the implicit function theorem, we deduce a family of
Hamiltonian deformations $\ell_{t,h_{t,u}}$ of $\ell_t$ which are
relatively HSLAG with respect to the complex structures $u\cdot J$
for all $u\in G$. Equivalently, $u\cdot \ell_{t,h_{t,u}}$ is relatively
HSLAG with respect to $J$. For each $t$, there exists a non-degenerate local minimum
of the modified volume functional $\voltilde$. Since it is non degenerate, we may choose
$u_t$, depending smoothly on $t$ such that $u_t\cdot \ell_{t,h_{t,u_t}}$
achieve such a local minimum of $\voltilde$. 

By Proposition~\ref{prop:critical},
$u_t\cdot \ell_{t,h_{t,u_t}}:L\to M$ must be HSLAG with respect to $J$.
We deduce the following proposition:
\begin{prop}
\label{prop:almostsolution}
Let $\ell_t:L\to M$ be a $G_\ell^o$-invariant HSLAG toric fibration in
a Kähler manifold $(M,\omega, J_0)$ such that $\ell_0$ is rigid.

For each positive almost complex structure $J$ compatible with $\omega$ and
sufficiently close to $J_0$, there exists a smoothly varying family of Hamiltonian transformations
$v_t$ such that $v_t\circ \ell_t$ is HSLAG with respect to $(M,\omega, J)$.
\end{prop}

Since we have a local smooth fibration, we readily deduce
\begin{cor}
Under the assumptions of Proposition~\ref{prop:almostsolution},
  there exists Hamiltonian transformations $v$ such that
$v\circ \ell_t$ is a HSLAG toric fibration with respect to $(M,\omega, J)$.
\end{cor}

This proves Theorem \ref{theo:B}.

\vspace{10pt}
\bibliographystyle{abbrv}
\bibliography{hslag}

\end{document}